\documentclass{amsart}

\usepackage{geometry}
\usepackage{amsfonts,amsmath,amssymb,amsthm}

\theoremstyle{plain} \newtheorem{cor}{Corollary}[section]
\theoremstyle{plain} 
\theoremstyle{plain} \newtheorem{thm}[cor]{Theorem}
\theoremstyle{definition} \newtheorem{defn}[cor]{Definition}
\theoremstyle{plain} 
\theoremstyle{plain}  
\theoremstyle{definition} 
\theoremstyle{plain}  
\theoremstyle{definition} 
\theoremstyle{definition} 
\theoremstyle{definition} 
\theoremstyle{definition} 
\theoremstyle{definition} \newtheorem{ex}[cor]{Example}

\newcommand{\btablesize}{\begin{small}}
\newcommand{\etablesize}{\end{small}}
\newcommand{\Zn}{\mathbb Z_n}
\newcommand{\Z}{\mathbb Z}
\newcommand{\C}{\mathbb C}
\newcommand{\dom}{\textup{dom}}
\newcommand{\ran}{\textup{ran}}
\newcommand{\rk}{\textup{rk}}
\newcommand{\Y}{\mathcal Y}
\newcommand{\D}{\mathcal D}
\newcommand{\N}{\mathbb N}
\newcommand{\ip}[1]{\left\langle #1 \right\rangle}
\newcommand{\ld}{\lfloor}
\newcommand{\rd}{\rfloor}
\newcommand{\IRR}{\textup{IRR}}
\newcommand{\perm}{\textup{perm}}

\usepackage{lineno}

\begin{document}



\title{Inverse semigroup spectral analysis for partially ranked data}

\author{Martin E. Malandro}
\thanks{The author was partially supported by AFOSR under grant FA9550-06-1-0027.}

\email{malandro@shsu.edu}

\address{Box 2206, Department of Mathematics and Statistics, Sam Houston State University, Huntsville, TX 77341-2206, USA. Phone: 1-936-294-1580. Fax: 1-936-294-1882.}

\date{\today}

%

\begin{abstract}
Motivated by the notion of symmetric group spectral analysis developed by Diaconis, we introduce the notion of spectral analysis on the rook monoid (also called the symmetric inverse semigroup), characterize its output in terms of symmetric group spectral analysis, and provide an application to the statistical analysis of partially ranked (voting) data. We also discuss generalizations to arbitrary finite inverse semigroups. This paper marks the first non-group semigroup development of spectral analysis.
\end{abstract}



\maketitle

\section{Introduction}

Spectral analysis is a model-free, symmetry-based approach to the statistical exploration and description of datasets. In \cite{Persi}, P.\ Diaconis gives a method for the spectral analysis of partially ranked voting data using symmetric groups. Briefly, suppose an $n$-candidate election is being held and each voter is asked to rank as many candidates as he wishes in order of preference, from position $1$ (best) to position $k$ (worst among the candidates the voter chooses to rank).
His technique begins by viewing this collection of partial rankings as elements of certain symmetric group modules and taking orthogonal projections of these elements onto the minimal invariant subspaces of these modules. Such projections give a complete, non-redundant description of the dataset. To provide easily understood statistics, inner products of these projections with certain ``easily interpretable" functions are computed. The resulting statistics constitute the {\em spectral analysis} (or the {\em symmetric group spectral analysis}) of the dataset. His approach was the first non-abelian finite group generalization of the usual Fourier (or spectral) analytic techniques of time series, based on the abelian group $\Zn$.

While the symmetric group acts naturally on the set of all partial rankings of $n$ objects, the rook monoid $R_n$ {\em is} this set of partial rankings. In this paper we explain the algebra of the rook monoid and generalize the notion of symmetric group spectral analysis to this new setting. Complications arise because the rook monoid is not a group, and while rook monoid modules decompose into minimal invariant subspaces, they are not necessarily orthogonal under the natural inner product. We define the notion of {\em rook monoid spectral analysis}, resolve these complications, and give a complete description of its output in terms of symmetric group spectral analysis. 

The main contributions of the present work are the following.
First, the algebra $\C R_n$ provides a natural framework for the encoding and analysis of partially ranked data for {\em all} partial rankings, not just rankings of objects in positions $1$ through $k$.
Second, we define and analyze two approaches to rook monoid spectral analysis. 
%
Under the {\em groupoid basis association} (defined in Section \ref{SecGroupoidBasisRook}), we show in Theorem \ref{ThmMainThm} that rook monoid spectral analysis offers a more local, granular approach to the statistical analysis of partially ranked (but not fully ranked) data than symmetric group spectral analysis does, in that it amounts to the partitioning of 
a dataset
by rank, domain, and range, followed by symmetric group spectral analysis using appropriately-sized symmetric groups on each piece of the partition. We reinforce this with an example in Section \ref{SecRookExample}.
Under the {\em semigroup basis association} (defined in Section \ref{SecRepDefs}), we show in Theorem \ref{ThmSpectralUnderSemigroupBasisAssoc} that rook monoid spectral analysis offers a hierarchical approach to the statistical analysis of partially ranked data, in that 
it amounts to, for each pair of subsets $D$ and $R$ of $\{1,2,\ldots,n\}$ such that $|D|=|R|$,
a symmetric group spectral analysis, using an appropriately-sized symmetric group, of the partial rankings in the dataset whose domains extend $D$ and whose ranges extend $R$.
Finally, we discuss generalizations to other semigroups. 

We proceed as follows.
In Section \ref{SecPartialRankings} we review basic facts about partial rankings, the rook monoid, and inverse semigroups. 
In Section \ref{SecRepDefs} we use basic ideas from the representation theory of inverse semigroups to define the Fourier transform on a finite inverse semigroup and we explain how this definition gives rise to two important examples of group-based spectral analysis---time series and the symmetric group spectral analysis of Diaconis.

In Section \ref{SecRookSpectral} we extend group-based spectral analysis to the rook monoid and we discuss extensions to inverse semigroups in general. In Section \ref{SecPreliminaries} we discuss how to perform symmetric group spectral analysis using appropriately-sized symmetric groups on each part of the partition (by rank, domain, and range) of a dataset consisting of partial rankings and we organize the goals for the rest of the paper. In Section \ref{SecGroupoidBasisRook} we review the groupoid basis of $\C R_n$ and we define the groupoid basis association. In Section \ref{SecOrthogIsotypic} we describe an inner product under which the isotypic subspaces of $\C R_n$ are mutually orthogonal, we describe easily interpretable functions for partially ranked data, and we describe the isotypic subspaces of $\C R_n$ in terms of the natural statistical information they carry. We develop rook monoid spectral analysis under the groupoid basis association and characterize the statistics it generates in Section \ref{SecRookSpectralGroupoidBasisAssoc}, and we look at an example in Section \ref{SecRookExample}. We then consider rook monoid spectral analysis under the semigroup basis association and characterize the statistics it generates in Section \ref{SecRookSpectralSemigroupBasisAssoc}. In Section \ref{SecNaturalInnerProd} we look at what happens if we try to use the natural inner product instead of the inner product introduced in Section \ref{SecOrthogIsotypic}. Section \ref{SecConclusion} contains thoughts on directions for future research.

Our development relies on several results from the representation theory of inverse semigroups, which are included in appendices.
In 
Appendix \ref{AppendixBasicRepTheory} 
we review the basic definitions from the representation theory of inverse semigroups. In 
Appendix \ref{SecGroupoidBasis} 
we review results of B.\ Steinberg \cite{Steinberg2} on the groupoid basis of an inverse semigroup algebra and the decomposition of an inverse semigroup algebra into a direct sum of matrix algebras over group algebras. Our results in Section \ref{SecOrthogIsotypic} and our proofs of Theorems \ref{ThmMainThm} and \ref{ThmSpectralUnderSemigroupBasisAssoc} are based on theorems for inverse semigroup algebras in general, which we state and prove in 
Appendix \ref{SecAppendixIsotypics}.


\section{Partial rankings}
\label{SecPartialRankings}

Suppose a five-candidate election is being held and you, as a voter, are asked to rank as many of these five candidates as you wish in any positions. The candidates are labeled 1 through 5. If you prefer candidate 4 in first position, candidate 5 in second, candidate 1 in third, candidate 3 in fourth, and candidate 2 in fifth, your vote would be
\[
\pi=\left(
\begin{array}{ccccc}
1&2&3&4&5\\
3&5&4&1&2\\
\end{array}
\right),
\]
a permutation on $\{1,2,3,4,5\}$. We will write our permutations with the domain on the top row and corresponding images on the bottom row, so here, for example, $\pi(3)=4$. If, on the other hand, you wished to express the same preference as above for candidates 2, 4, and 5, without ranking candidates 1 and 3, your vote would be
\[
\sigma=\left(
\begin{array}{ccccc}
1&2&3&4&5\\
-&5&-&1&2\\
\end{array}
\right),
\]
a {\em partial ranking} on $\{1,2,3,4,5\}$. The dashes in the second row indicate that 1 and 3 are not in the domain of $\sigma$. The ranking $\sigma$ indicates a preference for candidates 4 and 5 and a strong distaste for candidate 2, without committing to a ranking of the intermediate candidates 1 and 3. 

\begin{defn}
A {\em partial ranking $\sigma$ on $\{1,2,\ldots,n\}$} is an injective partial function from $\{1,2,\ldots,n\}$ to $\{1,2,\ldots,n\}$. The {\em domain} of $\sigma$, denoted $\dom(\sigma)$, is the set of elements $k\in\{1,2,\ldots,n\}$ for which $\sigma(k)$ is defined. The {\em range} of $\sigma$, denoted $\ran(\sigma)$, is $\{\sigma(k):k\in\dom(\sigma)\}$. The {\em rank} of $\sigma$, denoted $\rk(\sigma)$, is $|\dom(\sigma)|=|\ran(\sigma)|$.
\end{defn}

For the choice of $\sigma$ above, we have $\dom(\sigma)=\{2,4,5\}$, $\ran(\sigma)=\{1,2,5\}$, and $\rk(\sigma)=3$. 
We use the usual operation of partial function composition and we adopt the convention that maps act on the left of sets and are composed right-to-left: If $\sigma, \gamma$ are partial rankings on $\{1,2,\ldots,n\}$, then $\sigma\circ\gamma$ is the partial ranking on $\{1,2,\ldots,n\}$ whose domain is the set of elements $k$ for which $k\in \dom(\gamma)$ and $\gamma(k)\in\dom(\sigma)$, and if $k\in \dom(\sigma\circ\gamma),$ then $(\sigma\circ\gamma)(k)=\sigma(\gamma(k))$.

A {\em semigroup} is a nonempty set with an associative binary operation. A {\em monoid} is a semigroup with an identity element. Unless otherwise specified, we will write our semigroup operations multiplicatively.

\begin{defn}For an integer $n\geq 0$, the {\em rook monoid} $R_n$ is the set of all partial rankings on $\{1,2,\ldots,n\}$ under the operation of partial function composition.
\end{defn}

It is easy to check that $R_n$ is a monoid. We denote the symmetric group on $\{1,2,\ldots,n\}$ by $S_n$. $R_n$ contains $S_n$ as the set of elements of rank $n$, and the identity element for the operation on $R_n$ is the identity of $S_n$. In fact, $R_n$ contains isomorphic copies of all $S_k$, for $0\leq k\leq n$, which are obtained by identifying $S_k$ with the set of elements of $R_n$ whose domain and range are both $\{1,2,\ldots,k\}$.

\begin{defn}If $S$ and $T$ are semigroups, then a {\em semigroup homomorphism} (or just {\em homomorphism}) $\phi$ from $S$ to $T$ is a map $\phi:S\rightarrow T$ such that $\phi(ab)=\phi(a)\phi(b)$ for all $a,b\in S$. A {\em semigroup isomorphism} (or just {\em isomorphism}) is a semigroup homomorphism that is one-to-one and onto.
\end{defn}

Another way to view $R_n$ is as the set of all $n\times n$ matrices that have {\em at most} one 1 in each row and column (their other entries being 0), under the usual operation of matrix multiplication. Such matrices are called {\em rook matrices}. Given a partial ranking $\sigma\in R_n$, we can create an $n\times n$ rook matrix by placing a $1$ in the $i,j$ position whenever $\sigma(j)=i$ and having all other entries be 0. It is clear that this association is an isomorphism, and furthermore that the rank of a partial ranking $\sigma$ is the same as the rank of its associated rook matrix. $R_n$ is called the rook monoid because the collection of $n\times n$ rook matrices corresponds to the set of possible placements of non-attacking rooks on an $n\times n$ chessboard.

Although $R_n$ is not a group (unless $n=0$, in which case $R_0\cong \Z_1$), $R_n$ does have a nice algebraic structure---that of an inverse semigroup \cite{CliffPres}.

\begin{defn}An {\em inverse semigroup} is a semigroup $S$ with the property that, for each $x\in S$, there exists a {\em unique} $y\in S$ such that $xyx=x$ and $yxy=y$. In this case, $y$ is said to be the {\em inverse} of $x$, and we write $x^{-1}=y$.
\end{defn}

It follows that, in an inverse semigroup, if $x^{-1}=y$ then $y^{-1}=x$, $xx^{-1}$ is idempotent, and if $e$ is idempotent then $e^{-1}=e$.
Every group is an inverse semigroup, but not conversely. Also, we have emphasized the word {\em unique} in this definition, as uniqueness of an element's inverse does not follow from the rest of the hypotheses as it does for groups. For example, for $n\geq 2$, in $T_n$, the {\em full transformation semigroup on $n$ elements} (the set of all functions from $\{1,2,\ldots,n\}$ to $\{1,2,\ldots,n\}$ under function composition), for each element $x$ there is at least one $y$ such that $xyx=x$ and $yxy=y$, and there exist elements $x$ for which there are multiple elements $y$ satisfying both equations. 

It is easy to see that the inverse of an element $\sigma\in R_n$ is the partial ranking $\gamma$ whose domain is $\ran(\sigma)$, and whose definition (informally) is given by sending everything in $\ran(\sigma)$ back where it came from. Viewing the elements of $R_n$ as rook matrices, the inverse of a rook matrix is its transpose. 
%

\section{Representations and spectral analysis}


\label{SecRepDefs}

Our development of spectral analysis depends on the representation theory of inverse semigroups. The basic definitions are similar to those for groups, and are included in 
Appendix \ref{AppendixBasicRepTheory} 
for the convenience of the reader. Let $S$ be a finite inverse semigroup and let $\C S$ denote the complex algebra of $S$. 

\begin{defn}
The natural basis of $\C S$, i.e., the basis $\{s\}_{s\in S}$, is called the {\em semigroup basis} of $\C S$.
\end{defn}

Elements of $\C S$ can be identified with complex-valued functions on $S$ in a natural way. Specifically, if $f:S\rightarrow \C$, then $f$ corresponds to the element $\sum_{s\in S}f(s)s \in \C S$. $\C S$ can therefore be seen as the algebra of complex-valued functions on $S$. This association between functions on $S$ and elements of $\C S$ is called the {\em semigroup basis association}. There is, for non-group inverse semigroups in general, a different natural basis of $\C S$ and therefore another natural way to associate functions on $S$ and elements of $\C S$, called the {\em groupoid basis association}, which we define in Section \ref{SecGroupoidBasisRook}.

$\C S$ is semisimple. When $S$ is a group, this is Maschke's theorem \cite{DennisFarb}. For general $S$, this is a result of Munn \cite[Theorem 4.4]{Munn}. 
Since $\C S$ is semisimple, Wedderburn's theorem applies to $\C S$. Semisimplicity and the Wedderburn isomorphism are the key ingredients for the spectral analysis we develop.

If $f\in \C S$ and $\rho$ is a matrix representation of $\C S$, denote $\rho(f)$ by $\hat f(\rho)$.

\begin{thm}[Wedderburn's theorem] 
\label{ThmWedderburn}
Let $\Y$ be a complete set of inequivalent, irreducible matrix representations of $\C S$. Then $\Y$ is finite, and the map
\begin{equation}
\label{EqWedderburn}
\bigoplus_{\rho\in\Y}:\C S\rightarrow \bigoplus_{\rho\in\Y} M_{d_\rho}(\C)
\end{equation}
is an isomorphism of algebras. Explicitly, if $f\in \C S$, with $f=\sum_{s\in S}f(s)s$, then
\[
f\mapsto \bigoplus_{\rho\in\Y}\hat f(\rho) = \bigoplus_{\rho\in\Y}\sum_{s\in S}f(s)\rho(s)
\]
in this isomorphism.
\end{thm}

\begin{defn}Given $f\in \C S$ and a complete set of inequivalent, irreducible matrix  representations $\Y$ of $\C S$, the {\em Fourier transform of $f$ according to $\Y$} (or just the Fourier transform of $f$) is the image of $f$ in the Wedderburn isomorphism (\ref{EqWedderburn}). 
\end{defn}

\begin{defn}The inverse image of the natural basis of the algebra on the right in the Wedderburn isomorphism (\ref{EqWedderburn}) (that is, the set of matrices in this algebra which have a 1 in one position and 0 in all other positions)  is called the {\em Fourier basis of $\C S$ according to $\Y$.}
\end{defn}

Thus the Fourier transform of $f$ is, in general, a block diagonal matrix with complex entries, and we can view the Fourier transform of $f$ according to $\Y$ as a change of basis within $\C S$, from the natural basis $\{s\}_{s\in S}$ of $\C S$, to the Fourier basis of $\C S$ according to $\Y$.

\label{SecSpectralDefs}

Fourier transforms are closely related to the notions of spectral analysis. We begin by seeing how these notions 
apply to time series.

\begin{ex}[Time series] 
\label{ExTimeSeries}
Let $S=(\Zn,+)=\{0,1,\ldots,n-1\}$, the cyclic group of order $n$.
The irreducible representations of $\C \Z_n$ are all one-dimensional---they are the characters $\chi_k$ for $k\in\{0,\ldots,n-1\}$, defined on the natural basis of $\C \Z_n$ by $\chi_k(t)=e^{{2\pi i k t}/{n}}$.
If $f:\Z_n\rightarrow \C$ and we view $f\in \C \Z_n$ as $f=\sum_{t=0}^{n-1}f(t)t$, then
\[
\hat f(\chi_k)=\sum_{t=0}^{n-1}f(t)e^{{2 \pi i k t}/{n}},
\]
the familiar discrete Fourier transform of $f$. The Fourier basis of $\C \Zn$ is the usual basis of sampled exponentials $\{b_k\}_{k=0}^{n-1}\subset \C \Zn$:
\[
b_k= \frac{1}{n}\sum_{t=0}^{n-1}e^{{-2\pi i k t}/{n}}t.
\]
\end{ex}

We now explain how things generalize beyond $S=\Zn$. We can often view a dataset as an element of some $\C S$-module for some finite inverse semigroup $S$. Let $S$ be a finite inverse semigroup and let $f$ be a dataset, viewed in some way as an element of some left $\C S$-module $M$. Since $\C S$ is semisimple, $M$ decomposes into a direct sum of irreducible $\C S$-submodules $M_i$:
\[
M=M_1\oplus M_2\oplus \cdots \oplus M_k.
\]
Unfortunately, the $M_i$ are not uniquely determined in general. For any irreducible submodule $N$ of $M$, whether or not $N$ appears as a direct summand in this particular decomposition, let $V_N$ denote the sum of all irreducible submodules of $M$ isomorphic to $N$. $V_N$ is called the {\em isotypic component of $M$ of type (or isomorphism class) $N$}. As $N$ ranges across the irreducible submodules of $M$, we obtain the {\em isotypic components} $V_N$ of $M$. They are uniquely determined, and $M$ decomposes as the direct sum of them. Furthermore, given any decomposition of $M$ into irreducibles
$
M=M_1\oplus M_2\oplus \cdots \oplus M_k,
$
if we group the $M_i$ according to their isomorphism classes and sum together the $M_i$ from each isomorphism class, then we obtain the isotypic components of $M$ \cite{Chevalley}. Let
\[
M=V_1\oplus V_2\oplus \cdots \oplus V_m
\]
be the decomposition of $M$ into its isotypic components $V_1, V_2,\ldots, V_m$. This decomposition is more crude, in general, than a decomposition of $M$ into irreducibles, but it has the advantage of being a unique decomposition of $M$ into invariant subspaces under the action of $\C S$. In fact, it is the finest unique decomposition of $M$ into invariant subspaces under the action of $\C S$, in the sense that attempting to decompose any $V_i$ further into invariant subspaces requires a choice of basis. We do not want our definition of spectral analysis to depend on an arbitrary choice such as this, so it is the isotypic decomposition that we will work with. Loosely speaking, the {\em spectral analysis} of $f\in M$ is the examination of the projections of $f$ onto the isotypic components $V_i\subseteq M$. We call these projections the {\em isotypic projections of $f$}. 

In the case of time series, $\C \Z_n$ decomposes into a sum of $n$ one-dimensional isotypic components---with notation as in Example \ref{ExTimeSeries}, let $V_k=\C\textup{-span}(b_k) = \C\textup{-span}(\sum_{t=0}^{n-1}e^{{-2\pi i k t}/{n}}t)\subseteq \C \Z_n$. Then we have the isotypic decomposition $\C \Z_n=\bigoplus_{k=0}^{n-1}V_k$, and spectral analysis of $f\in \C \Z_n$ amounts to an examination of the projections of $f$ onto the $V_k$. In contrast to time series, however, many of the $V_i$ may be multidimensional in general, and to make the notion of spectral analysis precise for a given semigroup $S$ we will need a method to extract information from these projections. How exactly we should do this depends on the particular semigroup under consideration. We will explain this for the symmetric group in Example \ref{ExPersi}, and for the rook monoid in Sections \ref{SecOrthogIsotypic} and \ref{SecRookSpectralGroupoidBasisAssoc}.

The most important $\C S$-module is $M=\C S$ itself (where the action of $\C S$ on $M=\C S$ is given by the multiplication of $\C S$), where the isotypic decomposition of $M$ can be obtained from the Wedderburn isomorphism (\ref{EqWedderburn}). $\C S$ is both a left and right $\C S$-module. Notice that the inverse image of a column (respectively, row) of the $\rho$ block of the algebra on the right in (\ref{EqWedderburn}) is an irreducible left (respectively, right) submodule of $\C S$ of isomorphism class $\rho$. The inverse image of the $\rho$ block in (\ref{EqWedderburn}) is thus the isotypic component of $M$ of isomorphism class $\rho$, and is also a minimal two-sided ideal of $\C S$. Hence the isotypic decomposition of $\C S$ is the same as the (unique) decomposition of $\C S$ into the direct sum of its minimal two-sided ideals. 

Isotypic projections in $\C S$ are easy to compute from Fourier transforms. Let $\rho$ be an irreducible matrix representation of $\C S$ and let $\Y$ be any set of inequivalent, irreducible matrix representations of $\C S$. Let $\gamma\in \Y$ denote the representation in $\Y$ equivalent (if not equal) to $\rho$, and denote the isotypic component of $\C S$ of type $\rho$ by $V_\rho$. To compute the isotypic projection of $f\in \C S$ onto $V_\rho$, take the Fourier transform of $f$ according to $\Y$, set all coefficients of the result equal to 0 except for the ones in the $\gamma$ block, and take the inverse image of that. The result is 
the isotypic projection of $f$ onto $V_\rho$. It is easy to see that this works regardless of the particular matrix representations chosen for $\Y$. Computationally efficient methods for computing Fourier transforms and their inverses on a wide variety of groups and semigroups have been developed. See, for example, \cite{Baum,Clausen,CooleyTukey,RookFFT,InvSemiFFT,Maslen,DanDiameters,DanSepVars,DanWreath}.

\begin{ex}[Symmetric group spectral analysis] 
\label{ExPersi}
This example is an exposition of the ideas of Diaconis \cite{Persi}. 
We explain his ideas from an algebraic standpoint that will be useful for us when we generalize to the rook monoid in Section \ref{SecRookSpectral}. It can be shown that our development here is equivalent (in the sense that it generates the same statistics for any partially ranked voting dataset on any number of candidates) to his. We review only the algebraic aspects that generate the statistics. For full discussion, including a large example and inferential issues, see \cite{Persi}.

First we explain his technique as applied to fully ranked votes. A collection of votes in which every voter ranks each of $n$ candidates in order of preference defines a $\C$-valued (actually, a $\Z$-valued) function on $S_n$, where $f(\sigma)$ is the number of voters casting a ballot of type $\sigma$. Let $f:S_n\rightarrow \C$ and view $f$ as an element of $\C S_n$ as $f=\sum_{\sigma\in S_n}f(\sigma)\sigma$. 
There is a well-known bijection between the irreducible representations of $\C S_n$ and the partitions of $n$ \cite{JamesAndKerber}, so we write
\[
\C S_n = \bigoplus_{\lambda\vdash n}V^{\lambda}
\]
where $V^\lambda$ is the isotypic subspace corresponding to the irreducible representation for the partition $\lambda$. The irreducible representation for $\lambda$ is commonly described in terms of the action of $S_n$ on tableaux of shape $\lambda$---see, e.g., \cite{JamesAndKerber}. 
What we really need for spectral analysis are combinatorial descriptions of the $V^\lambda$ themselves. The descriptions we give below are due to Diaconis \cite{Persi}, and will allow us to describe the natural statistical information each isotypic subspace $V^\lambda$ carries. 

There is a natural inner product on $\C S_n$ given by
\[
\ip{f,g} = \ip{\sum_{\sigma\in S_n}f(\sigma)\sigma,\sum_{\sigma\in S_n}g(\sigma)\sigma} = \sum_{\sigma\in S_n}f(\sigma)\overline{g(\sigma)}.
\]
Under this inner product, the isotypic subspaces of $\C S_n$ are mutually orthogonal \cite[Chapter 2]{Serre}. We now project $f \in \C S_n$ onto each subspace. That is, we write
$
f = \sum_{\lambda \vdash n}f^\lambda,
$ 
for unique elements $f^\lambda \in V^\lambda$. These projections $f^\lambda$ may be computed by running a (fast) Fourier transform on $S_n$, provided $n$ is not too large \cite{Clausen, Maslen}. Other projection formulas are also available. See, for example, \cite[Theorem 1]{Persi}, \cite{DanCharacterProjection}, and \cite[Theorem 8]{Serre}.

Next, we examine the projections $f^\lambda$. This is analogous to examining the component frequencies of a function in the $S=\Zn$ case. However, in our case, many of the $V^\lambda$ are multidimensional, and in addition to concrete descriptions of these spaces we will use an additional device (which Diaconis attributes to C.\ Mallows) to extract information from the projections onto these spaces. 

We also note that, in a similar fashion to how the frequencies of highest amplitude carry the most information about the structure of a continuous waveform, here 
the lengths of the projections are important in determining which projections carry the most information about the structure of a dataset. However, due to the differences in dimensionality between the isotypic subspaces involved here, it is sometimes appropriate in making this determination to weight the lengths of the projections based on the dimensions of the subspaces in which they reside---see \cite{Persi} for more details. By orthogonality of isotypic subspaces, we have
\begin{equation}
\label{eqSumOfSquares}
||f||^2=\ip{f,f} = \sum_{\lambda \vdash n}\ip{f^\lambda,f^\lambda} = \sum_{\lambda \vdash n} ||f^\lambda||^2,
\end{equation}
which allows one to compute and compare easily the lengths of the projections.

First, $V^{(n)}$ is the space of constant functions on the fully ranked votes. It is one-dimensional, and 
\[
f^{(n)} = \left(\sum_{\sigma\in S_n}f(\sigma)\right) \left(\frac{1}{n!}\sum_{\sigma\in S_n}\sigma \right),
\]
where the quantity on the right is the Fourier basis element (for {\em any} complete set of inequivalent, irreducible representations of $\C S_n$) lying in $V^{(n)}$. The projection $f^{(n)}$ therefore records the number of votes cast.

Next, there are $n^2$ {\em easily interpretable} (or just {\em interpretable}) first-order functions. They are of the form
\[
\delta_{i \mapsto j} = \sum_{\sigma\in S_n}\delta_{i \mapsto j}(\sigma)\sigma,
\]
where
\[
\delta_{i \mapsto j} (\sigma) = 
\begin{cases}1 & \textup{if }\sigma(i)=j, \cr
0 & \textup{otherwise,}
\end{cases}
\]
as $i$ and $j$ range over $\{1,2,\ldots,n\}$. $V^{(n-1,1)}$ is an $(n-1)^2$-dimensional space. A general element of $V^{(n-1,1)}$ has the form
\[
\sum_{i,j}a_{i,j}\delta_{i \mapsto j} 
\]
where, since $V^{(n-1,1)}$ is orthogonal to $V^{(n)}$, $\sum_{i,j} a_{i,j} = 0$. $V^{(n-1,1)}$ carries the  ``pure" first-order statistics for fully ranked votes (i.e., the first-order information about the data once the average---the zeroth-order information---has been removed by the projection onto $V^{(n)}$). The device of Mallows used by Diaconis for extracting information from $f^{(n-1,1)}$ is this \cite[Section 2C]{Persi}: for each $i,j\in \{1,2,\ldots,n\},$ examine the inner product of $f^{(n-1,1)}$ with $\delta_{i \mapsto j}$. It turns out that 
\[
\ip{f^{(n-1,1)},\delta_{i\mapsto j}} = \sum_{\sigma \in S_n}f(\sigma)w(\sigma),
\]
where
\[
w(\sigma)=
\begin{cases}
\frac{n-1}{n} & \textup{if }\sigma(i)=j, \\
-\frac{1}{n} & \textup{otherwise.}
\end{cases}
\] 

Next, just as there are easily interpretable first-order functions, there are also easily interpretable second-order (ordered and unordered) functions. The easily interpretable second-order unordered functions are the
\[
\delta_{\{i_1,i_2\} \mapsto \{j_1,j_2\}} = 
\sum_{\sigma\in S_n}\delta_{\{i_1,i_2\} \mapsto \{j_1,j_2\}}(\sigma) \sigma,
\]
where
\[
\delta_{\{i_1,i_2\} \mapsto \{j_1,j_2\}}(\sigma) = 
\begin{cases}
1 &\textup{if } \{\sigma(i_1), \sigma(i_2)\} = \{j_1,j_2\}, \cr
0 & \textup{otherwise}.
\end{cases}
\]
The representation theory of $\C S_n$ implies that every element of $V^{(n-2,2)}$ is a linear combination of the $\delta_{\{i_1,i_2\} \mapsto \{j_1,j_2\}}$ which is orthogonal to the other isotypic subspaces. We denote the easily interpretable second-order ordered functions (defined analogously) by
$
\delta_{i_1 \mapsto j_1 , i_2 \mapsto j_2},
$ 
and elements of $V^{(n-2,1,1)}$ are linear combinations of such which are orthogonal to the other isotypic subspaces. There are also easily interpretable third-order functions and so on. As before, we compute the inner products of $f^{(n-2,2)}$ with the $\delta_{\{i_1,i_2\} \mapsto \{j_1,j_2\}}$ and of $f^{(n-1,1,1)}$ with the $\delta_{i_1 \mapsto j_1 , i_2 \mapsto j_2}$ to obtain second-order statistics of $f$.

In a similar fashion, we can continue as far as we'd like with the remaining isotypic subspaces to extract third-order and higher-order statistics about $f$.

\begin{defn}
The statistics created by projecting a data vector $f\in\C S_n$ onto the isotypic subspaces of $\C S_n$ and computing the inner products of these projections with the easily interpretable functions as described above constitute the {\em symmetric group spectral analysis} of $f$.
\end{defn}


Next we define Diaconis's notion of symmetric group spectral analysis for partially ranked votes. Let $k\leq n$ and suppose we are interested in analyzing the set of votes in an election with $n$ candidates in which every voter ranks their top $k$ candidates in order of preference. The collection of such votes defines a function $f$ on the rank-$k$ elements of $R_n$ of range $\{1,2,\ldots,k\}$, where $f(\sigma)$ is the number of voters who prefer the partial ranking $\sigma$. For each 
element $\sigma$ of $R_n$ of range $\{1,2,\ldots,k\}$, form the following element of $\C S_n$:
\[
\sigma' = \frac{f(\sigma)}{E(\sigma)}\sum_{\substack{t\in S_n: t\geq \sigma}}t,
\]
where $t\geq \sigma$ simply means that $t$ extends $\sigma$ as a partial function, and $E(\sigma)$ is the number of elements $t \in S_n$ that extend $\sigma$. Next, form the following element of $\C S_n$:
\[
F = \sum_{\sigma\in R_n: 
\ran(\sigma)=\{1,2,\ldots,k\}}\sigma'.
\]
Finally, compute the symmetric group spectral analysis of $F$. If a dataset of partial rankings contains data consisting of multiple ranks, then the analysis begins by separating the data according to rank and then proceeds separately, rank-by-rank, generating a different set of statistics for the data of each rank. The main example in \cite{Persi} consists of data of ranks one through five.

\end{ex}

\section{Rook monoid spectral analysis}
\label{SecRookSpectral}

\subsection{Preliminaries}
\label{SecPreliminaries}

As explained in Example \ref{ExPersi}, symmetric group spectral analysis begins by partitioning a dataset of partial rankings by rank before analyzing it---the output of the rank-$k$ spectral analysis for voting data depends only on the rank-$k$ votes. For voting data, it might make sense to partition the dataset by rank before performing spectral analysis if one thinks that voters who vote with different ranks might vote differently. Indeed, this was the case in the main example in \cite{Persi}.

For certain kinds of voting (or other partially ranked) data it might not make sense to partition the data by rank before analyzing it (in which case the full dataset can be averaged to create an element of $\C S_n$ for analysis), or it might make sense to partition the data to an even finer degree before analyzing it---for instance, one might partition the data by rank, domain, and range before performing symmetric group spectral analysis on each part of the partition.

To explain what we mean by symmetric group spectral analysis on such a set of partially ranked data, fix $k\leq n$ and two subsets $D$ and $R$ of $\{1,2,\ldots,n\}$ of size $k$, and let $R_n^{D,R}=\{\sigma\in R_n:\dom(\sigma)=D,\ran(\sigma)=R\}$. Suppose we wish to perform symmetric group spectral analysis on a function $f:R_n^{D,R}\rightarrow \C$. Let $p_D,p_R\in R_n$ be the unique order preserving bijections from $\{1,2,\ldots,k\}$ to $D$ and $R$, respectively. Identify $S_k$ with the elements of $R_n$ whose domain and range are both $\{1,2,\ldots,k\}$, view $f$ as an element of $\C S_k$ by
\[
f=\sum_{\sigma\in R_k^{D,R}} f(\sigma) ({p_R}^{-1} \sigma p_D),
\]
and apply symmetric group spectral analysis (in $\C S_k$) to $f$. As an example, for $n=5$, $k=3$, $D=\{2,4,5\}$, and $R=\{1,2,5\}$, for the choice of $\sigma$ at the beginning of Section \ref{SecPartialRankings} we have
\[
{p_R}^{-1}\sigma p_D = 
\left(\begin{array}{ccccc}
1&2&3&4&5\\
3&1&2&-&-
\end{array}
\right).
\]
Effectively, this is because $\sigma$ sends the first element of its domain to the third element of its range, the second element of its domain to the first element of its range, and the third element of its domain to the second element of its range.

We now proceed by extending the ideas of symmetric group spectral analysis to the rook monoid. In fact, for a function $f:R_n\rightarrow \C$, we define two different approaches to the rook monoid spectral analysis of $f$. (There are two different natural bases of $\C R_n$, and the differences in our approaches arise from the choice of which basis to associate with the delta functions of the elements of $R_n$.) Under the {\em groupoid basis association}, defined in Section \ref{SecGroupoidBasisRook}, we show in Theorem \ref{ThmMainThm} that the rook monoid spectral analysis of $f$ amounts to the idea for analysis above---that is, it amounts to a partitioning of $f$ by rank, domain, and range, before performing symmetric group spectral analysis (using appropriately-sized symmetric groups) on each part of the partition. Under the semigroup basis association, we show in Theorem \ref{ThmSpectralUnderSemigroupBasisAssoc} that rook monoid spectral analysis offers a hierarchical approach to the analysis of $f$---in particular, it is the same as the rook monoid spectral analysis, under the groupoid basis association, of the function $g:R_n\rightarrow \C$ given by
\[
g(\sigma) = \sum_{t\in R_n: t\geq \sigma}f(t),
\]
where $t\geq \sigma$ means that $t$ extends $\sigma$ as a partial function.

To explain precisely how the algebra of $\C R_n$ leads to these methods of analysis we will need the following three things. First, we need an inner product under which the isotypic subspaces of $\C R_n$ are mutually orthogonal. Second, we need easily interpretable functions for partially ranked data. Finally, we need concrete descriptions of the isotypic subspaces of $\C R_n$ in terms of the natural statistical information each carries. We handle these three considerations in Section \ref{SecOrthogIsotypic} with the help of the other natural basis of $\C R_n$---the {\em groupoid basis}---which we now review.

\subsection{The groupoid basis of the rook monoid algebra}
\label{SecGroupoidBasisRook}

The natural partial order on $R_n$ is defined in the following way: for $s,t\in R_n$, say $t\leq s$ if and only if $s$ extends $t$ as a partial function. The {\em groupoid basis} of $\C R_n$ is the collection $\{\ld s\rd\}_{s\in R_n}$, where
\[
\ld s\rd = \sum_{t\in R_n: t\leq s} (-1)^{\rk(s)-\rk(t)} t.
\]
It is well known \cite{Stanley, Steinberg2} that $(-1)^{\rk(s)-\rk(t)}=\mu(t,s)$, where $\mu$ is the M\"obius function of $\leq$, so we can recover the semigroup basis of $\C R_n$ by inverting the M\"obius function:
\[
s=\sum_{t\in R_n: t\leq s}\ld t \rd.
\]
The groupoid basis is a basis for $\C R_n$, with multiplication given by the following formula \cite{Steinberg2}:
\[
\ld s \rd \ld t \rd = \begin{cases}
\ld st \rd & \textup{if } \dom(s)=\ran(t), \\
0 & \textup{otherwise}.
\end{cases}
\]
That is, the product $\ld s \rd \ld t \rd$ is nonzero in $\C R_n$ precisely when the domain of $s$ lines up exactly with the range of $t$.

There is a corresponding notion of a groupoid basis $\{\ld s \rd\}_{s\in S}$ of $\C S$ for any finite inverse semigroup $S$ \cite{Steinberg2}, which we review in 
Appendix \ref{SecGroupoidBasis}. 

\begin{defn} Let $S$ be a finite inverse semigroup and let $f:S\rightarrow \C$. Under the {\em groupoid basis association}, $f$ corresponds to the element $\sum_{s\in S}f(s)\ld s \rd \in \C S$.
\end{defn}

It turns out that the groupoid basis 
of $\C S$ has a number of important implications for the representation theory of $\C S$---for us, it will be instrumental in describing the isotypic subspaces of $\C S$ and it will also yield an inner product under which the isotypic subspaces of $\C S$ are mutually orthogonal. We describe these implications for $\C R_n$ in Section \ref{SecOrthogIsotypic} and for $\C S$ in general in 
Appendix \ref{SecAppendixIsotypics}.
\subsection{Isotypic subspaces, interpretable functions, and an orthogonal inner product}

\label{SecOrthogIsotypic}

In this section we give an inner product under which the isotypic subspaces of $\C R_n$ are mutually orthogonal, we describe easily interpretable functions for partially ranked data, and we give concrete descriptions of the isotypic subspaces of $\C R_n$ in terms of the natural statistical information they carry.

We begin by noting that under the {\em natural inner product} on $\C R_n$ (obtained by declaring the semigroup basis of $\C R_n$ orthonormal), the isotypic subspaces of $\C R_n$ are not mutually orthogonal in general. 
For a simple example, consider $\C R_1=\C\textup{-span} (\textup{Id},N)$ (where $N$ denotes the null map). The irreducible representations of $\C R_1$ are both 1-dimensional. They are given by the linear extension of $\rho^0(x)=1$ for all $x\in R_1$, and by the linear extension of
\[
\rho^1(\textup{Id})=1, \quad \rho^1(N)=0.
\]
$\C R_1$ therefore splits into isotypics as $\C R_1=V^1\oplus V^0$, where $V^0=\C\textup{-span}(N)$ and $V^1=\C\textup{-span}(\textup{Id}-N)$. Under the natural inner product on $\C R_n$, we see that $\ip{N,\textup{Id}-N}=-1$, so $V^0$ and $V^1$ are not mutually orthogonal. This failure is caused by an ``entanglement" between ranks that increases as $n$ increases. Although an inner product under which the isotypic subspaces are mutually orthogonal is not strictly  necessary for spectral analysis, it would give us nice mathematical properties (for instance, if $f,g\in \C S$ and $\bar f, \bar g$ denote the projections of $f$ and $g$ onto some isotypic subspace of $\C S$, then under such an inner product we would have $\ip{f,\bar g} = \ip{\bar f,g} = \ip{\bar f, \bar g}$), and it would aid in a sum-of-squares analysis as in (\ref{eqSumOfSquares}). The groupoid basis effectively undoes the entanglement between ranks that causes this failure---declaring the groupoid basis orthonormal yields an inner product under which the isotypic subspaces of $\C R_n$ are mutually orthogonal.

\begin{thm}
\label{ThmIPOrthogRn}
Let $\ip{\cdot,\cdot}$ be the sesquilinear form on $\C R_n$ induced by, for $s,t\in R_n$,
\[
\ip{\ld s \rd,\ld t \rd} = \begin{cases}
1 & \textup{if } s=t, \\
0 & \textup{otherwise.}\end{cases}
\]
Then, with respect to this inner product, the isotypic subspaces of $\C R_n$ are mutually orthogonal.\end{thm}

Theorem \ref{ThmIPOrthogRn} was proved in \cite{RookFFT}. We extend it to finite inverse semigroups in general in Theorem \ref{OrthogInnerProductS}. 

Next we describe {\em rank-$k$ easily interpretable functions} (or just {\em interpretable functions}) for partially ranked data. Let $k\leq n$. The zeroth-order interpretable functions are the functions $\delta^{D,R}: R_n\rightarrow \C$, defined by
\[
\delta^{D,R}(\sigma) = 
\begin{cases}
1 & \textup{if } \dom(\sigma) = D \textup{ and } \ran(\sigma) = R, \\
0 & \textup{otherwise},
\end{cases}
\]
as $D$ and $R$ range across the size-$k$ subsets of $\{1,2,\ldots,n\}$. If $k\geq 1$ we also have the first-order interpretable functions $\delta^{D,R}_{i\mapsto j}$, defined by
\[
\delta^{D,R}_{i\mapsto j}(\sigma) = 
\begin{cases}
1 & \textup{if }\dom(\sigma)=D,\textup{ } \ran(\sigma)=R, \textup{ and }\sigma(i)=j, \\
0 & \textup{otherwise,} 
\end{cases}
\]
as $D$ and $R$ range over the size-$k$ subsets of $\{1,2,\ldots,n\}$, $i$ ranges over $D$, and $j$ ranges over $R$. If $k\geq 2$ we also have second-order unordered and second-order ordered interpretable functions. The second-order unordered interpretable functions are the $\delta^{D,R}_{\{i_1,i_2\}\mapsto\{j_1,j_2\}}$, defined by
\[
\delta^{D,R}_{\{i_1,i_2\}\mapsto\{j_1,j_2\}}(\sigma)=
\begin{cases}
1 &\textup{if } \dom(\sigma)=D,\textup{ } \ran(\sigma)=R, \textup{ and }\{\sigma(i_1), \sigma(i_2)\} = \{j_1,j_2\}, \cr
0 & \textup{otherwise},
\end{cases}
\]
as $D$ and $R$ range over the size-$k$ subsets of $\{1,2,\ldots,n\}$, $\{i_1, i_2\}$ ranges over the size-2 subsets of $D$, and $\{j_1, j_2\}$ ranges over the size-2 subsets of $R$. The second-order ordered interpretable functions are defined similarly.  If $k\geq 3$ we also have third-order interpretable functions which are defined in an analogous fashion, and so on.

Next we describe the isotypic subspaces of $\C R_n$. The isotypic subspaces of $\C R_n$ are in bijection with the partitions of the integers $\{0,1,\ldots,n\}$ (which can be seen from Theorem \ref{ThmRepGeneration}), so write
\[
\C R_n = \bigoplus_{k=0}^n \bigoplus_{\lambda\vdash k}V^\lambda,
\]
where $V^\lambda$ is the isotypic subspace of $\C R_n$ corresponding to the irreducible representation for the partition $\lambda$. The irreducible representation corresponding to the partition $\lambda$ can be described by combining descriptions of the irreducible representations of the symmetric group with Theorem \ref{ThmRepGeneration}, and the following descriptions of the $V^\lambda$ arise by combining Diaconis's descriptions of the isotypic subspaces of $\C S_n$ in Example \ref{ExPersi} and \cite{Persi} with a technical result in 
Appendix \ref{SecAppendixIsotypics} 
(Theorem \ref{FourierBasisSThm}). Let $k\leq n$. 

$V^{(k)}$ is spanned by the elements
\begin{equation}
\label{EqVkSpan}
\sum_{\sigma\in R_n}\delta^{D,R}(\sigma)\ld \sigma \rd,
\end{equation}
as $D$ and $R$ range over the size-$k$ subsets of $\{1,2,\ldots,n\}$. $V^{(k)}$ therefore carries zeroth-order information for rank-$k$ data. Notice that, for any fixed choice of $D,R$, the element in (\ref{EqVkSpan}) is the function $\delta^{D,R}$ viewed as an element of $\C R_n$ under the groupoid basis association.

For purposes of the following descriptions, let us continue to view the interpretable functions as elements of $\C R_n$ under the groupoid basis association. For example, we have
\[
\delta^{D,R}_{i\mapsto j} = \sum_{\sigma\in R_n} \delta^{D,R}_{i\mapsto j}(\sigma) \ld \sigma \rd.
\]
Every element of $V^{(k-1,1)}$ is of the form
\[
\sum_{\substack{D,R\subseteq \{1,2,\ldots,n\}:\\|D|=|R|=k}}\sum_{i,j} a_{i,j}^{D,R}\delta_{i\mapsto j}^{D,R},
\]
where for every choice of $D$ and $R$ we have $\sum_{i,j}a_{i,j}^{D,R}=0$. $V^{(k-1,1)}$ therefore carries pure first-order information for rank-$k$ data. 

Similarly, every element of $V^{(k-2,2)}$ is a linear combination of the $\delta_{\{i_1,i_2\}\mapsto\{j_1,j_2\}}^{D,R}$ which is orthogonal to the other isotypic subspaces, and so on. $V^{(k-2,2)}$ therefore carries pure second-order unordered information for rank-$k$ data, $V^{(k-2,1,1)}$ carries pure second-order ordered information for rank-$k$ data, and so on.

\subsection{Rook monoid spectral analysis under the groupoid basis association}
\label{SecRookSpectralGroupoidBasisAssoc}

Let $f:R_n\rightarrow \C$. We now define rook monoid spectral analysis of $f$ under the groupoid basis association, i.e., where we view $f\in \C R_n$ by
\[
f=\sum_{\sigma\in R_n}f(\sigma)\ld \sigma \rd.
\]
We use the inner product on $\C R_n$ induced by declaring the groupoid basis mutually orthogonal. As in Section \ref{SecOrthogIsotypic}, let us view easily interpretable functions as elements of $\C R_n$ under the groupoid basis association.

\begin{defn} Let $f:R_n\rightarrow \C$. The statistics created by projecting $\sum_{\sigma\in R_n}f(\sigma)\ld \sigma \rd$ onto the isotypic subspaces of $\C R_n$ and computing the inner products of these projections with the appropriately-paired easily interpretable functions constitute the {\em rook monoid spectral analysis of $f$ under the groupoid basis association}.
\end{defn}

By {\em appropriately-paired}, we mean that the easily interpretable functions $\delta^{D,R}$ with $|D|=|R|=k$ are paired with the projection $f^{(k)}$, the interpretable functions $\delta^{D,R}_{i\mapsto j}$ with $|D|=|R|=k$ are paired with the projection $f^{(k-1,1)}$, the interpretable functions $\delta^{D,R}_{\{i_1,i_2\}\mapsto \{j_1,j_2\}}$ with $|D|=|R|=k$ are paired with the projection $f^{(k-2,2)}$, and so on.

We now describe the statistics that result from this approach in terms of symmetric group spectral analysis. For every pair of size-$k$ subsets $D$ and $R$ of $\{1,2,\ldots,n\}$, denote by $f^{D,R}$ the restriction of $f$ to $\{\sigma\in R_n:\dom(\sigma)=D,\ran(\sigma)=R\}$. We may apply symmetric group spectral analysis (using $\C S_k$) to $f^{D,R}$ in the manner described in Section \ref{SecPreliminaries}.

\begin{thm}
\label{ThmMainThm}
The statistics generated by the rook monoid spectral analysis of $f$, under the groupoid association, are the same as the statistics generated by applying symmetric group spectral analysis using appropriately-sized symmetric groups, in the manner described in Section \ref{SecPreliminaries}, separately to each function in the collection $$\bigcup_{k=0}^{n}\{f^{D,R}: D,R\subseteq \{1,2,\ldots,n\},|D|=|R|=k\}.$$
\end{thm}

By {\em appropriately-sized symmetric groups}, we simply mean that $\C S_k$ is used for $f^{D,R}$ when $k=|D|=|R|$.

\begin{proof}[Proof of Theorem \ref{ThmMainThm}]
Let $k\leq n$ and let $D,R\subseteq \{1,2,\ldots,n\}$ with $|D|=|R|=k$. Let $f=\sum_{\sigma\in R_n}f(\sigma)\ld \sigma\rd$. Let $F^{D,R}:R_n\rightarrow \C$ by
\[
F^{D,R}(\sigma)=\begin{cases}
f(\sigma) & \textup{if }\dom(\sigma)=D\textup{ and }\ran(\sigma)=R,\\
0 & \textup{otherwise}.
\end{cases}
\]
For clarity, the only difference between $F^{D,R}$ and $f^{D,R}$ are their domains.
View $F^{D,R}$ as an element of $\C R_n$ under the groupoid basis association. Theorem \ref{ThmUnderstandingProjections} says that every non-zero isotypic projection of $F^{D,R}$ in $\C R_n$ can be written in terms of groupoid basis elements $\ld \sigma \rd$ for which $\dom(\sigma)=D$ and $\ran(\sigma)=R$, and that the non-zero isotypic projections of $F^{D,R}$ in $\C R_n$ are (after perhaps a relabeling of the domain and range) the same as the non-zero isotypic projections of $f^{D,R}$ in $\C S_k$. 
The easily interpretable functions of domain $D$ and range $R$ in $\C R_n$ were defined in such a way that their inner products with the isotypic projections of $F^{D,R}$ in $\C R_n$ are the same as the inner products of the isotypic projections of $f^{D,R}$ in $\C S_k$ with the easily interpretable functions in $\C S_k$. Furthermore, it is immediate that the inner products of the isotypic projections of $F^{D,R}$ in $\C R_n$ with the other interpretable functions in $\C R_n$ are zero.

Next, if $g=\sum_{\sigma\in R_n}g(\sigma) \ld \sigma \rd \in \C R_n$ and $g(\sigma)=0$ whenever $\dom(\sigma)=D$ and $\ran(\sigma)=R$, Theorem \ref{ThmUnderstandingProjections} says that the isotypic projections of $g$ in $\C R_n$, when written in terms of the groupoid basis, have nonzero coefficients only for groupoid basis elements $\ld \sigma \rd$ for which $\dom(\sigma)\neq D$ or $\dom(\sigma)\neq R$. Therefore, for any such element $g\in \C R_n$, the statistics generated by the inner products of the isotypic projections of $F^{D,R}$ with the easily  interpretable functions of domain $D$ and range $R$ in $\C R_n$ are the same as the statistics generated by the inner products of the isotypic projections of $F^{D,R}+g$ with the easily  interpretable functions of domain $D$ and range $R$ in $\C R_n$. In particular, for some such element $g\in \C R_n$ we have $f=F^{D,R}+g$, so the statistics that arise from the inner products of the isotypic projections of $f\in \C R_n$ with the interpretable functions of domain $D$ and range $R$ in $\C R_n$ are the same as the statistics that arise from the symmetric group spectral analysis (using $\C S_k$) of $f^{D,R}$.
\end{proof}

In an analogous fashion, Theorems \ref{FourierBasisSThm} and \ref{ThmUnderstandingProjections} show that if $S$ is any finite inverse semigroup, $f:S\rightarrow \C$, and we view $f\in \C S$ using the groupoid basis association, then spectral analysis of $f$ boils down to the spectral analysis of the components of $f$ using the group algebras of the maximal subgroups of $S$.


Theorem \ref{ThmMainThm} shows that rook monoid spectral analysis under the groupoid basis association is different from symmetric group spectral analysis, in that it gives a more granular picture of the partially ranked data in a dataset. Next we give a direct comparison between the two approaches in the context of an example.

\subsection{An example}
\label{SecRookExample}

We now apply rook monoid spectral analysis under the groupoid basis association to a particular collection of partially ranked voting data. Our dataset for this example is the well-studied collection of votes from the 1980 American Psychological Association (APA) election, in which voters were asked to rank five candidates in order of preference. 
{\mbox{15449}} people voted, of which 5738 fully ranked all five candidates. The rank-3 votes are tallied in Table \ref{rank3} \cite[Table 6]{Persi} and the votes of other ranks can be found in \cite[Tables 1 and 6]{Persi}. Each vote is a partial ranking.

\begin{table}[ht]
\btablesize
\caption[APA election: Rank-3 ballots]{Rank-3 ballots}
\begin{center}
\begin{tabular}{|cc|cc|cc|cc|}
\hline
\label{rank3}
Vote & Tally & Vote & Tally & Vote & Tally & Vote & Tally\\
\hline
$[1,2,3,-,-]$&27&$[3,1,-,-,2]$&38&$[1,-,-,2,3]$&44&$[-,3,1,-,2]$&16\\
$[1,3,2,-,-]$&79&$[2,3,-,-,1]$&35&$[1,-,-,3,2]$&35&$[-,2,3,-,1]$&14\\
$[2,1,3,-,-]$&31&$[3,2,-,-,1]$&41&$[2,-,-,1,3]$&46&$[-,3,2,-,1]$&15\\
$[3,1,2,-,-]$&32&$[1,-,2,3,-]$&30&$[2,-,-,3,1]$&62&$[-,1,-,2,3]$&55\\
$[2,3,1,-,-]$&83&$[1,-,3,2,-]$&21&$[3,-,-,1,2]$&90&$[-,1,-,3,2]$&45\\
$[3,2,1,-,-]$&57&$[2,-,1,3,-]$&39&$[3,-,-,2,1]$&75&$[-,2,-,1,3]$&54\\
$[1,2,-,3,-]$&19&$[3,-,1,2,-]$&15&$[-,1,2,3,-]$&9&$[-,3,-,1,2]$&62\\
$[1,3,-,2,-]$&22&$[2,-,3,1,-]$&15&$[-,1,3,2,-]$&17&$[-,2,-,3,1]$&50\\
$[2,1,-,3,-]$&31&$[3,-,2,1,-]$&13&$[-,3,1,2,-]$&26&$[-,3,-,2,1]$&59\\
$[3,1,-,2,-]$&45&$[1,-,3,-,2]$&41&$[-,2,1,3,-]$&17&$[-,-,1,2,3]$&15\\
$[2,3,-,1,-]$&28&$[1,-,2,-,3]$&49&$[-,2,3,1,-]$&21&$[-,-,1,3,2]$&19\\
$[3,2,-,1,-]$&51&$[2,-,1,-,3]$&74&$[-,3,2,1,-]$&18&$[-,-,2,1,3]$&16\\
$[1,2,-,-,3]$&26&$[3,-,1,-,2]$&47&$[-,1,2,-,3]$&8&$[-,-,3,1,2]$&46\\
$[1,3,-,-,2]$&31&$[2,-,3,-,1]$&37&$[-,1,3,-,2]$&15&$[-,-,2,3,1]$&17\\
$[2,1,-,-,3]$&17&$[3,-,2,-,1]$&32&$[-,2,1,-,3]$&16&$[-,-,3,2,1]$&20\\
\hline
\end{tabular}
\end{center}
\etablesize
\end{table}

This dataset defines a $\C$-valued (actually, a $\Z$-valued) function $f$ on $R_5$, where $f(\sigma)$ is the number of voters casting a ballot of type $\sigma$. The $\sigma$ here are written in standard list-form, with the top row removed. For example, looking at \cite[Table 1]{Persi} we have
\[
f\left(  
\begin{array}{ccccc}1&2&3&4&5 \cr 2&3&1&4&5\end{array}
 \right)
= 172,
\]
and from Table \ref{rank3} \cite[Table 6]{Persi} we see that
\[
f\left( 
\begin{array}{ccccc}1&2&3&4&5 \cr 3&-&-&2&1\end{array}
 \right)
= 75.
\]
We have the isotypic decomposition
\begin{align*}
\C R_5 =& \left(V^{(5)} \oplus V^{(4,1)} \oplus V^{(3,2)} \oplus V^{(3,1,1)} \oplus V^{(2,2,1)} \oplus V^{(2,1,1,1)} \oplus V^{(1,1,1,1,1)} \right)\\& 
\oplus \left(V^{(4)} \oplus V^{(3,1)} \oplus V^{(2,2)} \oplus V^{(2,1,1)} \oplus V^{(1,1,1,1)} \right) \\&
\oplus \left(V^{(3)} \oplus V^{(2,1)} \oplus V^{(1,1,1)} \right)\oplus 
\left(V^{(2)} \oplus V^{(1,1)} \right)\oplus 
\left(V^{(1)} \right)\oplus 
\left(V^{(0)}\right),
\end{align*}
where $V^\lambda$ is the isotypic subspace of $\C R_5$ corresponding to the irreducible representation for $\lambda$. View $f$ as an element of $\C R_5$ under the groupoid basis association.

We begin our analysis by projecting $f$ onto the isotypic subspaces, that is, by writing 
\[
f = \sum_{k=0}^5 \sum_{\lambda \vdash k}f^\lambda
\]
for unique elements $f^\lambda \in V^\lambda$. 
We use the inner product induced by declaring the groupoid basis elements of $\C R_n$ mutually orthonormal, so that the $V^\lambda$ are mutually orthogonal.

Under the groupoid basis association, the rank-$k$ data projects onto the $V^\lambda$ where $\lambda \vdash k$, and we may therefore carry out our analysis rank by rank. According to Theorem \ref{ThmMainThm}, our rank-5 analysis is exactly the same as that provided by symmetric group spectral analysis applied to the rank-5 votes. The results from that analysis may be found in \cite{Persi}. It is in the partially ranked data that rook monoid spectral analysis differs.

The projections $f^\lambda$ for $\lambda \vdash 4$ are all zero, as $f(\sigma)=0$ for all $\sigma\in R_5$ such that $\rk(\sigma)=4$. After all, ranking $n-1$ out of $n$ candidates naturally ranks the $n$th as well. 


Next we consider the projections $f^\lambda$ for $\lambda \vdash 3$. Recall from Section \ref{SecOrthogIsotypic} that $V^{(3)}$ is the sum of the spaces of constant functions for each of the rank-3 choices of domain and range. That is, $V^{(3)}$ is spanned by the elements
\[
\delta^{D,R} = \sum_{\sigma\in R_5}\delta^{D,R}(\sigma)\ld \sigma \rd, 
\]
as $D$ and $R$ range across all size-$3$ subsets of $\{1,2,3,4,5\}$. We take the projection $f^{(3)}$ and compute the inner products of it with these $\delta^{D,R}$ to obtain Table \ref{0thOrderRank3}. Notice that the $D,R$ entry is simply the number of rank-3 voters ranking the candidates in $D$ in the positions in $R$.

\begin{table}
\btablesize
\caption{Zeroth-order groupoid analysis, rank-3 data}
\begin{center}
\begin{tabular}{|c|ccccc|}
\hline
\label{0thOrderRank3}
& \multicolumn{5}{|c|}{{\bf Range}}\\
\hline
Domain&{\bf1,2,3}&{\bf1,2,4}&{\bf1,2,5}&$\cdots$&{\bf3,4,5}\\
\hline
1,2,3&309&0&0&$\cdots$&0 \\
1,2,4&196&0&0&$\cdots$&0 \\
1,2,5&188&0&0&$\cdots$&0 \\
1,3,4&133&0&0&$\cdots$&0 \\
1,3,5&280&0&0&$\cdots$&0 \\
1,4,5&352&0&0&$\cdots$&0 \\
2,3,4&108&0&0&$\cdots$&0 \\
2,3,5&84 &0&0&$\cdots$&0 \\
2,4,5&325&0&0&$\cdots$&0 \\
3,4,5&133&0&0&$\cdots$&0 \\
\hline
\end{tabular}
\end{center}
\etablesize
\end{table}

More interesting is $f^{(2,1)}$, which in this case contains both the pure first-order and second-order unordered information. To explain, we have the easily interpretable first-order rank-3 functions 
\[
\delta^{D,R}_{i\mapsto j} = \sum_{\sigma\in R_5}\delta^{D,R}_{i\mapsto j}(\sigma) \ld \sigma \rd,
\]
where $D, R$ are size-$3$ subsets of $\{1,2,3,4,5\}$, $i \in D$, and $j \in R$.
Every element of $V^{(2,1)}$ is of the form
\[
\sum_{D,R}\sum_{i,j}a^{D,R}_{i,j}\delta^{D,R}_{i\mapsto j},
\]
where, for every choice of $D,R$,
$
\sum_{i,j}a^{D,R}_{i,j} = 0.
$
When ranking three candidates, choosing a domain, a range, and the ranking of one of the candidates automatically defines the unordered set of rankings for the other two candidates. Thus, for the analogous second-order unordered rank-3 functions, we have (for $\{i_1,i_2\} \subset D, \{j_i,j_2\} \subset R$),
\[
\delta^{D,R}_{\{i_1,i_2\} \mapsto \{j_1,j_2\}} = \delta^{D,R}_{D \setminus \{i_1,i_2\} \mapsto R \setminus \{j_1,j_2\}}.
\]
$V^{(2,1)}$ therefore carries pure second-order unordered statistics as well.
%


Inner products of $f^{(2,1)}$ with the $\delta^{D,R}_{i\mapsto j}$ are given in Table \ref{1stOrderRank3Table1}. 
Entries in these tables have been rounded to two decimal places. By the comment above, the inner products of $f^{(2,1)}$ with the $\delta_{\{i_1,i_2\}\mapsto\{j_1,j_2\}}^{D,R}$ are just permutations of the entries in Table \ref{1stOrderRank3Table1}. For example, the inner products of $f^{(2,1)}$ with $\delta_{\{i_1,i_2\}\mapsto\{j_1,j_2\}}^{\{1,4,5\},\{1,2,3\}}$ are given in Table \ref{ExSecondOrderUnorderedRank3}. 

If we denote the rank-$3$ portion of $f$ by $f_3$,
\[
f_3 = \sum_{\sigma\in R_n: \rk(\sigma)=3}f(\sigma)\ld \sigma \rd,
\]
then we have
$
||f^{(3)} + f^{(2,1)} || > .996 ||f_3||,
$
so we discard the projection $f^{(1,1,1)}$ from our analysis.

For comparison, the results of symmetric group spectral analysis, as applied to the rank-3 portion of $f$, are given in Table \ref{DiaconisSecondOrderUnorderedRank3} \cite[Table 9]{Persi}.

%


\begin{table}
\btablesize
\caption{First-order groupoid analysis, rank-3 data}
\label{1stOrderRank3Table1}
\begin{minipage}[c]{0.55\linewidth}\centering
\begin{tabular}{|c|ccccc|}
\hline
\multicolumn{6}{|c|}{$D = \{1,2,3\}, R = \{1,2,3\}$}\\
\hline
& \multicolumn{5}{|c|}{{\bf Rank}}\\
\hline
Candidate & {\bf 1} & {\bf 2} & {\bf 3} &\bf{4} &\bf{5}\\
\hline
1 & 3 & 11 & -14 &0&0\\
2 & -40 & -19 & 59 &0&0\\
3 & 37 & 8 & -45 &0&0\\
4 & 0 & 0 & 0 &0&0\\
5 & 0 & 0 & 0 &0&0\\
\hline
\hline
\multicolumn{6}{|c|}{$D = \{1,2,5\}, R = \{1,2,3\}$}\\
\hline
& \multicolumn{5}{|c|}{{\bf Rank}}\\
\hline
Candidate & {\bf 1} & {\bf 2} & {\bf 3} &\bf{4} &\bf{5} \\
\hline
1&-5.67&-10.67&16.33&0&0\\
2&-7.67&4.33&3.33&0&0\\
3&0&0&0&0&0\\
4&0&0&0&0&0\\
5&13.33&6.33&-19.67&0&0\\
\hline
\hline
\multicolumn{6}{|c|}{$D = \{1,3,5\}, R = \{1,2,3\}$}\\
\hline
& \multicolumn{5}{|c|}{{\bf Rank}}\\
\hline
Candidate & {\bf 1} & {\bf 2} & {\bf 3} &\bf{4} &\bf{5} \\
\hline
1&-3.33&17.67&-14.33&0&0\\
2&0&0&0&0&0\\
3&27.67&-12.33&-15.33&0&0\\
4&0&0&0&0&0\\
5&-24.33&-5.33&29.67&0&0\\
\hline
\hline
\multicolumn{6}{|c|}{$D = \{2,3,4\}, R = \{1,2,3\}$}\\
\hline
& \multicolumn{5}{|c|}{{\bf Rank}}\\
\hline
Candidate & {\bf 1} & {\bf 2} & {\bf 3} &\bf{4} &\bf{5} \\
\hline
1&0&0&0&0&0\\
2&-10&2&8&0&0\\
3&7&-9&2&0&0\\
4&3&7&-10&0&0\\
5&0&0&0&0&0\\
\hline
\hline
\multicolumn{6}{|c|}{$D = \{2,4,5\}, R = \{1,2,3\}$}\\
\hline
& \multicolumn{5}{|c|}{{\bf Rank}}\\
\hline
Candidate & {\bf 1} & {\bf 2} & {\bf 3} &\bf{4} &\bf{5} \\
\hline
1&0&0&0&0&0\\
2&-8.33&-4.33&12.67&0&0\\
3&0&0&0&0&0\\
4&7.67&5.67&-13.33&0&0\\
5&0.67&-1.33&0.67&0&0\\
\hline
\end{tabular}
\end{minipage}
\begin{minipage}[c]{0.42\linewidth}\centering
\begin{tabular}{|c|ccccc|}
\hline
\multicolumn{6}{|c|}{$D = \{1,2,4\}, R = \{1,2,3\}$}\\
\hline
& \multicolumn{5}{|c|}{{\bf Rank}}\\
\hline
Candidate & {\bf 1} & {\bf 2} & {\bf 3}&\bf{4} &\bf{5}  \\
\hline
1&-24.33&-6.33&30.67&0&0\\
2&10.67&4.67&-15.33&0&0\\
3&0&0&0&0&0\\
4&13.67&1.67&-15.33&0&0\\
5&0&0&0&0&0\\
\hline
\hline
\multicolumn{6}{|c|}{$D = \{1,3,4\}, R = \{1,2,3\}$}\\
\hline
& \multicolumn{5}{|c|}{{\bf Rank}}\\
\hline
Candidate & {\bf 1} & {\bf 2} & {\bf 3} &\bf{4} &\bf{5} \\
\hline
1&6.67&9.67&-16.33&0&0\\
2&0&0&0&0&0\\
3&9.67&-1.33&-8.33&0&0\\
4&-16.33&-8.33&24.67&0&0\\
5&0&0&0&0&0\\
\hline
\hline
\multicolumn{6}{|c|}{$D = \{1,4,5\}, R = \{1,2,3\}$}\\
\hline
 & \multicolumn{5}{|c|}{{\bf Rank}}\\
\hline
Candidate & {\bf 1} & {\bf 2} & {\bf 3} &\bf{4} &\bf{5} \\
\hline
1&-38.33&-9.33&47.67&0&0\\
2&0&0&0&0&0\\
3&0&0&0&0&0\\
4&18.67&1.67&-20.33&0&0\\
5&19.67&7.67&-27.33&0&0\\
\hline
\hline
\multicolumn{6}{|c|}{$D = \{2,3,5\}, R = \{1,2,3\}$}\\
\hline
& \multicolumn{5}{|c|}{{\bf Rank}}\\
\hline
Candidate & {\bf 1} & {\bf 2} & {\bf 3} &\bf{4} &\bf{5} \\
\hline
1&0&0&0&0&0\\
2&-5&2&3&0&0\\
3&4&-5&1&0&0\\
4&0&0&0&0&0\\
5&1&3&-4&0&0\\
\hline
\hline
\multicolumn{6}{|c|}{$D = \{3,4,5\}, R = \{1,2,3\}$}\\
\hline
& \multicolumn{5}{|c|}{{\bf Rank}}\\
\hline
Candidate & {\bf 1} & {\bf 2} & {\bf 3} &\bf{4} &\bf{5} \\
\hline
1&0&0&0&0&0\\
2&0&0&0&0&0\\
3&-10.33&-11.33&21.67&0&0\\
4&17.67&-9.33&-8.33&0&0\\
5&-7.33&20.67&-13.33&0&0\\
\hline
\end{tabular}
\end{minipage}
\etablesize
\end{table}

\begin{table}[htb]
\btablesize
\caption{Second-order unordered groupoid analysis, votes with domain $\{1,4,5\}$ and range $\{1,2,3\}$}
\begin{center}
\label{ExSecondOrderUnorderedRank3}
\begin{tabular}{|c|ccc|}
\hline
& \multicolumn{3}{|c|}{{\bf Rank}}\\
\hline
Candidates & {\bf 1,2} & {\bf 1,3} & {\bf 2,3} \\
\hline
$1, 4$ & -27.33 & 7.67 & 19.67 \\
$1, 5$ & -20.33 & 1.67 & 18.67 \\
$4, 5$ & 47.67 & -9.33 & -38.33\\
\hline
\end{tabular}
\end{center}
\etablesize
\end{table}

\begin{table}[ht]
\btablesize
\caption{Diaconis's first-order and second-order unordered analysis, rank-3 data}
\label{DiaconisFirstOrderRank3}
\label{DiaconisSecondOrderUnorderedRank3}
\begin{minipage}[c]{0.45\linewidth}\centering
\begin{tabular}{|c|ccc|}
\hline
& \multicolumn{3}{|c|}{{\bf Rank}}\\
\hline
Candidate & {\bf 1} & {\bf 2} & {\bf 3} \\
\hline
1&2&76&114\\
2&-78&-28&52\\
3&2&-103&-116\\
4&38&-7&-48\\
5&35&63&-1\\
\hline
\end{tabular}
\end{minipage}
\hspace{0in}
\begin{minipage}[c]{0.45\linewidth}\centering
\begin{tabular}{|c|ccc|}
\hline
& \multicolumn{3}{|c|}{{\bf Rank}}\\
\hline
Candidates & {\bf 1,2} & {\bf 1,3} & {\bf 2,3} \\
\hline
1,2 & -50 & 6 & 12 \\
1,3 & 150 & -3 & -41 \\
1,4 & -71 & -8 & 11 \\
1,5 & -28 & 5 & 16 \\
2,3 & -2 & 24 & 28 \\
2,4 & 57 & -5 & -7 \\
2,5 & -5 & -24 & -34 \\
3,4 & -84 & -12 & -4\\
3,5 & -63 & -8 & 17\\
4,5 & 97 & 26 & 0\\
\hline
\end{tabular}
\end{minipage}
\etablesize
\end{table}

Through the examination of Tables \ref{1stOrderRank3Table1} through \ref{DiaconisFirstOrderRank3}, we see that rook monoid spectral analysis under the groupoid basis association offers a more local, granular inspection of the data than the symmetric group spectral analysis of Example \ref{ExPersi} does, in that it allows us to see how the natural subsets of the rank-$k$ voters vote amongst themselves. 

Positive numbers in these tables indicate a positive (larger-than-average) effect for choosing a candidate (or group of candidates) in a position (or group of positions), with higher values indicating stronger effects. A constant function would have all entries in these tables equal to 0 (as $V^{(3)}$ is orthogonal to $V^{(2,1)}$).
Negative values in these tables indicate lower-than-average effects, with larger-magnitude negative numbers indicating stronger negative effects. For example, the entry of $-38.33$ in the $((4,5),(2,3))$ position of Table \ref{ExSecondOrderUnorderedRank3} indicates that voters who chose to rank candidates 1, 4, and 5 (in positions 1, 2, and 3) ranked candidates 4 and 5 in positions 2 and 3 (without regard to order) considerably less often than would occur in a uniform spread of votes. A quick glance at the votes in Table \ref{rank3} reveals that the partial rankings $[1,-,-,2,3]$ and $[1,-,-,3,2]$ were indeed the two least-popular choices among the rank-3 voters ranking candidates 1, 4, and 5.

Before we proceed with an examination of the numbers in Table \ref{1stOrderRank3Table1}, we recall the main pattern that Diaconis found in the overall dataset \cite{Persi}. In the rank-5 data, Diaconis found large second-order unordered pair effects for ranking candidates 1 and 3 and candidates 4 and 5 in positions 1 and 2 and positions 4 and 5. He found similar patterns in the lower-order ranks in the dataset. The short story of this election was that candidates 1 and 3 were on one side, candidates 4 and 5 on the other, and candidate 2 was somewhere in the middle, a bit closer to candidates 4 and 5. Voters primarily tended to support one of these sets of candidates, either $\{1,3\}$ or $\{4,5\}$, and then chose between them. This is supported by the second-order data in Table \ref{DiaconisSecondOrderUnorderedRank3}. The first-order information in Table \ref{DiaconisFirstOrderRank3} also shows that among all rank-3 voters, candidates 4 and 5 were preferred overall. Note that the information in Table \ref{DiaconisFirstOrderRank3} accounts for the fact that in this election, for rank-3 voters, unranked candidates are implicitly ranked in one of the last two positions---this is automatically accounted for in the creation of this table because Diaconis's technique for symmetric group spectral analysis begins by averaging over missing data for partially ranked data.

Table \ref{1stOrderRank3Table1} allows us to proceed with a more local, granular examination of the rank-3 data. Examining these numbers  we see that, among the rank-3 voters, there is a consistent positive effect for ranking the candidate pairs (1, 3) and (4, 5) in positions 1 and 2 (without regard to order). We also see a positive first-order effect---and often a strong one---for ranking candidate 3 in position 1 {\em whenever candidate 3 is ranked, except when the other two candidates ranked are candidates 4 and 5}. These observations are consistent with the overall patterns in the dataset and further support the idea that candidates 1 and 3 were on one side of the election and candidates 4 and 5 were on the other side. We can also examine subsets of these tables together to glean further insights---for example, the numbers in Table \ref{1stOrderRank3Table1} show that among the rank-3 voters ranking both candidates 1 and 3, candidate 3 was heavily preferred. This information is not readily apparent from Table \ref{DiaconisFirstOrderRank3}. 


An examination of the projections $f^\lambda$ for $\lambda=2,1,0$ would proceed in a similar fashion, and is omitted.

\subsection{Rook monoid spectral analysis under the semigroup basis association}
\label{SecRookSpectralSemigroupBasisAssoc}

Let $f:R_n\rightarrow \C$. We now define rook monoid spectral analysis under the semigroup basis association, i.e., where we view $f$ as an element of $\C R_n$ by
\[
f = \sum_{\sigma\in R_n} f(\sigma)\sigma.
\]
Note that $f$, when expressed with respect to the groupoid basis, is
\[
\sum_{\sigma\in R_n} g(\sigma) \ld \sigma \rd,
\]
where
\[
g(\sigma) = \sum_{t \geq \sigma}f(t).
\]

We continue to use the inner product on $\C R_n$ induced by declaring the groupoid basis mutually orthonormal. We will remark on the alternative inner product (induced by declaring the semigroup basis mutually orthonormal) in Section \ref{SecNaturalInnerProd}.

As we did 
in Section \ref{SecRookSpectralGroupoidBasisAssoc}, to analyze $f$ we project $f$ onto the isotypic subspaces of $\C R_n$ and compute inner products with the appropriate easily interpretable functions. In general, let $E$ be an easily interpretable function for rank-$k$ data. Then $E:R_n\rightarrow \{0,1\}$ and $E(\sigma)=0$ if $\rk(\sigma)\neq k$. Let $E_G$ and $E_S$ denote $E$ viewed as an element of $\C R_n$ under the groupoid basis association and the semigroup basis association, respectively. We claim that we get the same numbers regardless of whether we compute inner products of the projections of $f$ with $E_G$ or $E_S$. To see this, let $\lambda\vdash k$. We have
\[
E_G = \sum_{\sigma\in R_n:\rk(\sigma)=k}E(\sigma)\ld \sigma \rd, \quad E_S = \sum_{\sigma\in R_n: \rk(\sigma)=k}E(\sigma) \sigma.
\]
Now, $E_S$, when expressed with respect to the $\ld \sigma \rd$ basis, is of the form
\[
\sum_{\sigma\in R_n:\rk(\sigma)=k}E(\sigma)\ld \sigma \rd + \sum_{\sigma\in R_n:\rk(\sigma) < k}z(\sigma)\ld\sigma\rd
\]
for some function $z$ on $R_n$. Also, the projection $f^\lambda$, when expressed in terms of the groupoid basis, contains nonzero coefficients only for elements $\ld \sigma \rd$ for which $\rk(\sigma)=k$. Therefore, $\ip{f^\lambda,E_G} = \ip{f^\lambda,E_S}$.

\begin{defn}Let $f:R_n\rightarrow \C$. The statistics created by projecting $\sum_{\sigma\in R_n}f(\sigma) \sigma$ onto the isotypic subspaces of $\C R_n$ and computing the inner products of these projections with the appropriately-paired easily interpretable functions constitute the {\em rook monoid spectral analysis of $f$ under the semigroup basis association}.
\end{defn}

By the discussion above, we have:

\begin{thm} 
\label{ThmSpectralUnderSemigroupBasisAssoc}
Let $f:R_n\rightarrow \C$. Then the statistics generated by the rook monoid spectral analysis of $f$ under the semigroup basis association are the same as the statistics generated by the rook monoid spectral analysis, under the groupoid basis association, of the function $g:R_n\rightarrow \C$, where
\[
g(\sigma) = \sum_{t\geq \sigma} f(t).
\]
\end{thm}

Theorem \ref{ThmSpectralUnderSemigroupBasisAssoc}
shows that rook monoid spectral analysis, under the semigroup basis association, offers a hierarchical approach to the spectral analysis of partially ranked data, where the rank-$k$ analysis is derived not just from the data of rank $k$, but instead from the data of rank $k$ and higher.

\subsection{Rook monoid spectral analysis under the natural inner product}
\label{SecNaturalInnerProd}

Let $f:R_n\rightarrow \C$. In this section we give a couple of remarks about what happens if we try to apply rook monoid spectral analysis to $f$ under the {\em natural inner product} on $\C R_n$, induced by declaring the semigroup basis of $\C R_n$ mutually orthonormal. As we saw in Section \ref{SecOrthogIsotypic}, under this inner product the isotypic subspaces of $\C R_n$ are not mutually orthogonal in general. This can interfere with a sum-of-squares analysis of the lengths of the projections of $f$, but as we are about to see, if we are careful then it will create the same easily-understood statistics as in Sections \ref{SecRookSpectralGroupoidBasisAssoc} and \ref{SecRookSpectralSemigroupBasisAssoc}. 

In particular, let $E$ be an easily interpretable function for rank-$k$ data. Then $E:R_n\rightarrow \{0,1\}$ and $E(\sigma)=0$ if $\rk(\sigma)\neq k$. Let $\ip{\cdot,\cdot}$ denote the inner product on $\C R_n$ induced by declaring the groupoid basis mutually orthonormal, and let $\ip{\cdot,\cdot}_s$ denote the natural inner product on $\C R_n$. View $f$ as an element of $\C R_n$ under the semigroup basis association or the groupoid basis association, and express $f$ with respect to the groupoid basis. So, if we are using the groupoid basis association, then
\[
f=\sum_{\sigma\in R_n}f(\sigma)\ld \sigma \rd,
\]
and if we are using the semigroup basis association, then
\[
f=\sum_{\sigma\in R_n}g(\sigma) \ld \sigma \rd,
\]
where $g(\sigma) = \sum_{t\geq \sigma} f(t).$ Let $\lambda$ be a partition of $k$ and consider the projection $f^\lambda$. If $E_G$ and $E_S$ denote $E$ viewed as an element of $\C R_n$ under the groupoid basis association and semigroup basis association, respectively, we saw in Section \ref{SecRookSpectralSemigroupBasisAssoc} that $\ip{f^\lambda,E_G} = \ip{f^\lambda,E_S}$. A similar argument 
shows that
$
\ip{f^\lambda,E_S}_s = \ip{f^\lambda,E_G},
$
so we have
\[
\ip{f^\lambda,E_S}_s = \ip{f^\lambda,E_G}=\ip{f^\lambda,E_S}.
\]
In this way we can use the natural inner product to perform rook monoid spectral analysis.

However, in general $\ip{f^\lambda,E_G}_s$ will be different due to interference from terms of rank lower than $k$ in the inner product, as will inner products such as $\ip{f,E_S^\lambda}_s$ and $\ip{f^\lambda,E_S^\lambda}_s$. Here is a simple example. Consider $f\in \C R_2$ by
\[
f= 1 \left(\begin{array}{cc}1&2\\1&2\end{array}\right)
+ 2 \left(\begin{array}{cc}1&2\\2&1\end{array}\right)
+ 4 \left(\begin{array}{cc}1&2\\1&-\end{array}\right)
+ 7 \left(\begin{array}{cc}1&2\\2&-\end{array}\right)
+ 6 \left(\begin{array}{cc}1&2\\-&1\end{array}\right) 
+ 3 \left(\begin{array}{cc}1&2\\-&-\end{array}\right).
\]

In the notation of Section \ref{SecOrthogIsotypic}, let $E=\delta^{\{1\},\{2\}}$, so $E_S=1\left(\begin{array}{cc}1&2\\2&-\end{array}\right)$. Also
\[
E_G = 1\left(\begin{array}{cc}1&2\\2&-\end{array}\right) - 1\left(\begin{array}{cc}1&2\\-&-\end{array}\right).
\]
Denote by $f^1$ and $E_S^1$ the projections of $f$ and $E_S$ onto $V^{(1)}$ (the only isotypic of $\C R_2$ of dimension greater than 1). We have
\[
f^1 = 5 \left(\begin{array}{cc}1&2\\1&-\end{array}\right)
+ 1 \left(\begin{array}{cc}1&2\\-&2\end{array}\right) 
+ 9 \left(\begin{array}{cc}1&2\\2&-\end{array}\right) 
+ 8 \left(\begin{array}{cc}1&2\\-&1\end{array}\right) 
-23 \left(\begin{array}{cc}1&2\\-&-\end{array}\right)
\]
and
\[
E_S^1=1 \left(\begin{array}{cc}1&2\\2&-\end{array}\right) 
-1 \left(\begin{array}{cc}1&2\\-&-\end{array}\right)=E_G.
\]
First, note that $\ip{f^1, E_G} = 9$, which tells us that there are 9 partial rankings in $f$ which map 1 to 2. We also have $\ip{f,E_S^1}_s=4$, which measures how many {\em more} rank-1 partial rankings in $f$ map 1 to 2 than there are null rankings in $f$. 
Also, $\ip{f^1,E_S^1}_s=32=\ip{f^1,E_G}_s$, which measures the total number of partial rankings in $f$ which map 1 to 2 {\em plus} the total number of partial rankings in $f$. In either case we can extract useful statistics (such as the number of rank-1 partial rankings in $f$ which map 1 to 2, or the total number of partial rankings in $f$ which map 1 to 2) by performing the appropriate additions and subtractions. However, such desirable statistics are already available to us, with no additional effort required, if we simply use the inner product $\ip{\cdot,\cdot}$ instead of the natural one. We obtain other benefits by using $\ip{\cdot,\cdot}$ as well, such as orthogonality of isotypic subspaces.

\section{Concluding remarks and open questions}
\label{SecConclusion}

As we have seen, if $S$ is a finite inverse semigroup, then via the groupoid basis of $\C S$, spectral analysis based on $S$ essentially boils down to group-based spectral analysis using the maximal subgroups of $S$.

The essential algebraic component that enables inverse semigroup spectral analysis seems to be the semisimplicity of $\C S$, which allows us to write any element $f$ of any $\C S$-module uniquely as the sum of its isotypic projections. For most finite semigroups $S$, $\C S$ is not semisimple. An intriguing example is $S=T_n$, the full transformation semigroup on $n$ elements. The question of what spectral analysis based on $S$ should mean in such a case remains an open question. In the simplest case, if $f\in \C S$ where $\C S$ is not semisimple, then the collection $\rho(f)$, as $\rho$ varies over a complete set of inequivalent, irreducible representations of $\C S$, does not uniquely determine $f$. What kinds of information do we lose if we only use the irreducible representations of $\C S$ and perform calculations with inner products as in this article? This is really a question about the Jacobson radical of $\C S$ and how it interacts with ``interpretable" functions based on $S$. If the irreducible representations are not enough to capture what statistics we want about a function, should we consider the indecomposable representations instead? And how should we go about understanding the spectral analysis of arbitrary functions (i.e., elements of arbitrary $\C S$-modules) if $S$ has infinite representation type, as is the case for $S=T_n$ for $n> 4$ \cite{TnRepType}? Furthermore, if $S$ is an arbitrary finite semigroup, is there a ``correct" notion of an inner product on $\C S$ in general, analogous to the inner product obtained by declaring the groupoid basis of $\C S$ orthonormal when $S$ is an inverse semigroup? As we have seen, sometimes the natural inner product is not the most useful one.

\section{Acknowledgments}
We thank the anonymous referees for their comments and suggestions, which have helped us improve the organization and presentation of this article.

\appendix
\section{Basic representation theory for inverse semigroups}
\label{AppendixBasicRepTheory}

In this appendix we review the basic definitions from the representation theory of inverse semigroups. These ideas carry over with little or no modification to general semigroups. For a treatment of the representation theory of semigroups in general, see \cite{Rhodes}. Let $S$ be a finite inverse semigroup.

\begin{defn}The {\em complex algebra of $S$}, denoted $\C S$, is (as a vector space) the $\C$-span of the symbols $s\in S$. The multiplication in $\C S$, called {\em convolution} and denoted by $\ast$, is defined by the linear extension of the multiplication in $S$ by the distributive law.
\end{defn}

If $S$ is a group, convolution may be written in the familiar way: If $f,g\in \C S$ with $f=\sum_{s\in S}f(s)s$ and $g=\sum_{s\in S}g(s)s$, then
\[
f\ast g = \sum_{s\in S}\sum_{r\in S} f(r)g(r^{-1}s)s.
\]

If $S$ has an identity $1_S$, then $\C S$ has a multiplicative identity, namely $1\cdot 1_S$. Even if $S$ does not have an identity, $\C S$ does. We can see this from the semisimplicity of $\C S$ \cite{Munn} and the Wedderburn isomorphism (Theorem \ref{ThmWedderburn})---the inverse image of the identity of the algebra on the right of (\ref{EqWedderburn}) is the identity of $\C S$. Denote the identity of $\C S$ by $1_{\C S}$ and the algebra of $n\times n$ matrices over $\C$ by $M_n(\C)$.

\begin{defn}A {\em matrix representation} (or just {\em representation}) $\rho$ of $\C S$ of dimension $d_\rho\in\N$ is a linear map $\rho:\C S\rightarrow M_{d_\rho}(\C)$ for which $\rho(ab)=\rho(a)\rho(b)$ for all $a,b\in \C S$, and for which $\rho(1_{\C S})$ is the identity matrix.
\end{defn}

Equivalently, a representation of $\C S$ is a finite-dimensional $\C$-vector space which is also a unital left $\C S$-module. In this paper we only consider finite-dimensional representations and unital modules, so left $\C S$-modules and representations of $\C S$ are the same.

\begin{defn}Matrix representations $\rho_1,\rho_2$ of $\C S$ are {\em equivalent} if there is an invertible matrix $A$ such that
\[
A\rho_1(x)A^{-1}=\rho_2(x)
\]
for all $x\in \C S$.
\end{defn}

That is, two representations are equivalent if they are isomorphic as left $\C S$-modules.

\begin{defn}A representation $\rho$ of $\C S$ is {\em irreducible} if it is simple as a left $\C S$-module. 
\end{defn}
Equivalently, $\rho$ is irreducible if there do not exist representations $\rho_1,\rho_2$, a matrix valued function $g$, and an invertible matrix $A$ for which
\[
A\rho(x)A^{-1}=\left[
\begin{array}{cc}
\rho_1(x)&0\\
g(x)&\rho_2(x)\\
\end{array}
\right]
\]
for all $x\in \C S$.

\begin{defn}$\C S$ is said to be {\em semisimple} if every left $\C S$-module is equal to a direct sum of simple left $\C S$-modules.
\end{defn}

\section{The groupoid basis of a finite inverse semigroup}

\label{SecGroupoidBasis}

Let $S$ be a finite inverse semigroup. There is a natural partial order on $S$ given by, for $s,t\in S$, $t\leq s$ if and only if $t=es$ for some idempotent $e\in S$ \cite{Mitsch}. Notice that if $S$ is a group, then this partial order is trivial in the sense that $t\leq s$ if and only if $t=s$. 
Recently B.\ Steinberg has used the M\"obius function of this partial order to realize the decomposition of $\C S$ into a direct sum of matrix algebras over group algebras \cite{Steinberg2}. To see how this works, we begin by reviewing the groupoid basis of $\C S$ \cite{Steinberg2}.

\begin{defn}The {\em groupoid basis} of $\C S$ is given by the collection $\{\ld s\rd\}_{s\in S}$, where
\[
\ld s\rd = \sum_{t\in S: t\leq s} \mu(t,s)t,
\]
and $\mu$ is the M\"obius function of the natural partial order on $S$.
\end{defn}

For $x,y\in R_n$, $x\leq y$ if and only if $y$ extends $x$ as a partial function. We can recover the semigroup basis of $\C S$ in terms of the groupoid basis by inverting the M\"obius function:
\[
s=\sum_{t\in S: t\leq s}\ld t \rd.
\]
The groupoid basis is a basis for $\C S$, whose multiplication is given by the following formula \cite{Steinberg2}:

\begin{equation}
\label{GroupoidMult}
\ld s \rd \ld t \rd = 
\begin{cases}
\ld st \rd & \textup{if } s^{-1}s=tt^{-1}, \\
0 & \textup{otherwise}.
\end{cases}
\end{equation}

For $s\in R_n$, $s^{-1}s$ is the partial identity on $\dom(s)$ and $ss^{-1}$ is the partial identity on $\ran(s)$ (keeping in mind that we view maps as acting on the left of sets and that we compose maps from right to left). It follows that for $s,t\in R_n$,
\[
\ld s \rd \ld t \rd = \begin{cases}
\ld st \rd & \textup{if } \dom(s)=\ran(t), \\
0 & \textup{otherwise}
\end{cases}
\]
in $\C R_n$. 

We will also need Green's $\D$-relation \cite{CliffPres, Green, Steinberg2}:

\begin{defn}Let $e,f\in S$ be idempotent. We say $e$ and $f$ are {\em isomorphic} if there is an element $s\in S$ such that $e=s^{-1}s$ and $f=ss^{-1}$. Idempotents $e$ and $f$ are said to be {\em $\D$-related} if they are isomorphic. In general, elements $s,t\in S$ are said to be {\em $\D$-related} if $s^{-1}s$ is isomorphic to $t^{-1}t$.
\end{defn}
The equivalence classes of $S$ under the $\D$-relation are the $\D$-classes of $S$. An equivalent characterization of $\D$ is that $s$ and $t$ are $\D$-related if and only if $s$ and $t$ generate the same two-sided ideal in $S$. 
For $R_n$, the idempotents are the restrictions of the identity map, and two idempotents are isomorphic if and only if they have the same rank. $R_n$ has $n+1$ $\D$-classes. They are $D_0,D_1,\ldots,D_n$, where $D_k$ is the set of elements of $R_n$ of rank $k$.

\begin{defn}A {\em subgroup} of $S$ is a subset of $S$ which is also a group. A subgroup $G$ of $S$ is {\em maximal} if $G$ is not contained in any other subgroup of $S$.
\end{defn}
Given an idempotent $e$ of $S$, there is precisely one maximal subgroup of $S$ containing $e$ \cite{CliffPres}, called {\em the maximal subgroup of $S$ at $e$} and denoted $G_e$. In fact \cite{Steinberg2}
\[
G_e = \{s\in S: s^{-1}s=ss^{-1}=e\},
\]
and $e$ is the identity of $G_e$. If $e$ and $f$ are isomorphic idempotents, it is straightforward to show that $G_e\cong G_f$. For $R_n$, the maximal subgroup at any idempotent $e$ of rank $k$ is isomorphic to $S_k$.

We can now describe Steinberg's decomposition of $\C S$ into a direct sum of matrix algebras over group algebras. Let $D_0,\ldots,D_n$ be the $\D$-classes of $S$. Let $\C D_k$ be the $\C$\textup{-}\textup{span} of $\{\ld s\rd:s\in D_k\}$. From (\ref{GroupoidMult}) it follows that $\C S=\bigoplus_{k=0}^n\C D_k$.
The following theorem can be found in \cite{Steinberg2}.

\begin{thm}
\label{BigThm}
Let $r_k$ indicate the number of idempotents in $D_k$, and let $e_k$ be any idempotent in $D_k$. Denote the maximal subgroup of $S$ at $e_k$ by $G_k$. Then there is an algebra isomorphism
$\phi: \C D_k \rightarrow M_{r_k}(\C G_k).$ 
\end{thm}
The isomorphism $\phi$ that Steinberg constructs to prove Theorem \ref{BigThm} is given explicitly as follows. For each $\D$-class $D_k$, fix an idempotent $e_k$. For every idempotent $a\in D_k$, fix an element $p_a\in S$ such that $p_a^{-1}p_a=e_k$ and $p_ap_a^{-1}=a$, taking $p_{e_k}=e_k$. It is straightforward to show and important to note that $p_a\in D_k$ (and hence $p_a^{-1}\in D_k$ as well). View the $r_k\times r_k$ matrices as being indexed by pairs of idempotents in $D_k$. Define $\phi$ on the basis $\{\ld s\rd:s\in D_k\}$ of $\C D_k$ in the following manner: for an element $\ld s\rd \in \C D_k$ with $s^{-1}s=e$ and $ss^{-1}=f$,
\[
\phi(\ld s\rd) = {p_f}^{-1}sp_{e} E_{f,e},
\]
where $E_{f,e}$ is the standard $r_k\times r_k$ matrix with a 1 in the $f,e$ position and 0 elsewhere. We have that ${p_f}^{-1}sp_e \in G_k$ by construction, the linear extension of $\phi$ to $\C D_k$ is the isomorphism, and the inverse of $\phi$ is induced by, for $s\in G_k$,
\[
sE_{f,e} \mapsto \ld p_f s {p_e}^{-1}  \rd.
\]
Thus $\C S\cong \bigoplus_{k=0}^n M_{r_k}(\C G_k)$.

This gives us the following powerful method for constructing the irreducible representations of $\C S$ from the irreducible representations of the maximal subgroups of $S$ \cite{Steinberg2}.

\begin{thm}
\label{ThmRepGeneration}
Let $D_0,D_1,\ldots,D_n$ be the $\D$-classes of $S$. For each $k\in \{0,1,\ldots,n\}$, fix an idempotent $e_k\in D_k$. Let $G_k$ be the maximal subgroup of $S$ at $e_k$, and let $\IRR(G_k)$ be any complete set of inequivalent, irreducible matrix representations for $G_k$. Then the irreducible representations of $\C S$ are in one-to-one correspondence with the elements of $\uplus_{k=0}^n \IRR(G_k)$. Specifically, given an irreducible representation $\rho$ of $\C G_k$, form the irreducible representation $\bar\rho$ of $M_{r_k}(\C G_k)$:
\[
\bar\rho(gE_{i,j}) = E_{i,j}\otimes \rho(g)
\]
for $g\in G_k$, where $E_{i,j}$ is the standard $r_k\times r_k$ matrix with a 1 in the $i,j$ position and $0$ elsewhere. Extend $\bar\rho$ linearly to the rest of $M_{r_k}(\C G_k),$ and further extend $\bar\rho$ to $\C S$ by letting it be $0$ on the other summands of $\bigoplus_{k=0}^nM_{r_k}(\C G_k)$. As $\rho$ ranges over $\uplus_{k=0}^n \IRR(G_k)$, the $\bar\rho$ form a complete set inequivalent, irreducible matrix representations of $\C S$.
\end{thm}

We will use this theorem to help us describe the isotypic subspaces of $\C S$ in terms of the isotypic subspaces of the complex algebras of the maximal subgroups of $S$ in 
Appendix \ref{SecAppendixIsotypics}. 
First, however, we explain what the isomorphism $\phi$ from Theorem \ref{BigThm} translates into when $S=R_n$. 

For a $\D$-class $D_k$ of $R_n$ (that is, the subset of elements of $R_n$ of rank $k$), let us take $e_k \in D_k$ to be the partial identity on $\{1,\ldots,k\}$, that is,
\[
e_k=\left(\begin{array}{ccccccc}
1&2&\cdots&k&k+1&\cdots&n \\
1&2&\cdots&k&-&\cdots&-
\end{array}\right).
\] 
We then have
\[
G_k=\{s\in R_n:\dom(s)=\ran(s)=\{1,2,\ldots,k\}\}.
\]
We identify $G_k$ with the permutation group $S_k$ in the obvious manner.

For an idempotent $a \in D_k$ (that is, a rank-$k$ restriction of the identity map), let us take $p_a$ to be the unique order-preserving bijection from $\{1,2,\ldots,k\}$ to $\dom(a)=\ran(a)$. For an element $s\in R_n$ of rank $k$, define the {\em permutation type} of $s$, $\perm(s)$, to be, informally, the ``arrows'' from $\dom(s)$ to $\ran(s)$, expressed as a permutation in $G_k = S_k$. For example, if 
\[
s=\left(\begin{array}{cccc}
1&2&3&4\\
4&-&1&2
\end{array}\right),
\textup{ then }
\perm(s) = \left(\begin{array}{ccc}
1&2&3\\
3&1&2
\end{array}\right)
\]
because $s$ sends the first element of its domain to the third element of its range, the second element of its domain to the first element of its range, and the third element of its domain to the second element of its range.

Formally, we define
\[
\perm(s)= {p_{ss^{-1}}}^{-1} s p_{s^{-1}s},
\]
where $p_{s^{-1}s}$ is the unique order preserving bijection from $\{1,2,\ldots,k\}$ to $\dom(s)$ and ${p_{ss^{-1}}}^{-1}$ is the unique order preserving bijection from $\ran(s)$ to $\{1,2,\ldots,k\}$.

The isomorphism $\phi$ from Theorem \ref{BigThm} now works as follows. We have $\binom{n}{k}\times \binom{n}{k}$ matrices, so let us index their rows and columns by the $k$-subsets of $\{1,2,\ldots,n\}$. We have
\[
\C D_k \cong M_{\binom{n}{k}}(\C S_k)
\]
where, if $s\in R_n$ has rank $k$, then $\phi(\ld s \rd) = \perm(s)E_{\ran(s),\dom(s)}$. This result was implicit in Munn's work on the rook monoid \cite{MunnCharacters}, and was first written down explicitly by Solomon \cite{Solomon}. Solomon's isomorphism is essentially the same as the one we just described. 
As a corollary, we have:
\begin{cor}
\label{CorRnIsomSks}
$\C R_n\cong \bigoplus_{k=0}^n M_{\binom{n}{k}}(\C S_k).$
\end{cor}
\section{Isotypic subspaces and an orthogonal inner product}
\label{SecAppendixIsotypics}

Let $S$ be a finite inverse semigroup. In this appendix we use the results of 
Appendix \ref{SecGroupoidBasis} 
to describe Fourier bases of $\C S$ and isotypic projections in $\C S$ in terms of those of the $\C G$, as $G$ ranges over the maximal subgroups of $S$, and we give an inner product on $\C S$ under which the isotypic subspaces of $\C S$ are mutually orthogonal.

As in 
Appendix \ref{SecGroupoidBasis},
let $D_0,\ldots,D_n$ be the $\D$-classes of $S$, let $r_k$ denote the number of idempotents in $D_k$, pick an idempotent $e_k$ in each $\D$-class $D_k$, and let $G_k$ be the maximal subgroup of $S$ at $e_k$. Then by Theorem \ref{ThmRepGeneration}, the isotypic subspaces of $\C S$ are in one-to-one correspondence with the isotypic subspaces of the $\C G_k$, as $k$ ranges from $0$ to $n$. Also, as in the isomorphism from Theorem \ref{BigThm}, for every idempotent $a\in D_k$, fix an element $p_a\in S$ such that ${p_a}^{-1}p_a=e_k$ and $p_a{p_a}^{-1}=a$ (and take $p_{e_k}=e_k$). Let $\IRR(G_k)$ be a complete set of inequivalent, irreducible matrix representations of $\C G_k$. For each $\rho \in \IRR(G_k)$, let $\bar \rho$ denote its extension (as in Theorem \ref{ThmRepGeneration}) to $\bigoplus_{k=0}^nM_{r_k}(\C G_k)$, and hence to $\C S$. Let $\Y=\{\bar\rho: \rho \in \uplus_{k=0}^n \IRR(G_k)\}$, so that $\Y$ is a complete set of inequivalent, irreducible matrix representations of $\C S$.

We begin to describe the isotypic subspaces of $\C S$ by describing the Fourier basis for $\C S$ according to $\Y$ in terms of Fourier bases of the $\C G_k$: If $B\subseteq \C S$ is the set of inverse images of the natural basis of $\bigoplus_{\bar \rho \in \Y}M_{d_{\bar \rho}}(\C)$ in the Wedderburn isomorphism 
\begin{equation}
\bigoplus_{\bar\rho \in \Y}\bar\rho:
\C S \rightarrow \bigoplus_{\bar\rho \in \Y}M_{d_{\bar \rho}}(\C)
\label{FourierIsomS}
,\end{equation}
then for each $y \in B$,
\[y = \sum_{s\in S}y(s)\ld s\rd. \]
We will describe the coefficients $y(s)$.

Suppose we already have an explicit description of a Fourier basis for $\C G_k$ for each $k \in \{0,\ldots,n\}$. That is, if $C$ is the set of inverse images of the natural basis of the algebra on the right in the isomorphism
\begin{equation}
\bigoplus_{\rho \in \IRR(G_k)}\rho:
\C G_k \rightarrow \bigoplus_{\rho \in \IRR(G_k)}M_{d_{\rho}}(\C),
\label{FourierIsomGk}
\end{equation}
then, for each $c\in C$,
\[
c = \sum_{x\in G_k}c(x)x.
\] 
We will describe the coefficients $y(s)$ in terms of the $c(x)$.

Let $\rho \in \IRR(G_k)$, and let $c_{i,j} \in \C G_k$,
\[
c_{i,j} = \sum_{x\in G_k}c_{i,j}(x)x,
\]
be the inverse image in the isomorphism (\ref{FourierIsomGk}) of the element of 
$
\bigoplus_{\rho \in \IRR(G_k)}M_{d_{\rho}}(\C)
$
that is $1$ in the $i,j$ position in the $\rho$ block and $0$ elsewhere. $\bar\rho$ maps to block matrices whose rows and columns are indexed by the idempotents in $D_k$, and whose entries are themselves $d_\rho \times d_\rho$ matrices. 
We have the following description of a Fourier basis for $\C S$, which generalizes the description for $\C R_n$ given in \cite{RookFFT}.

\begin{thm}
\label{FourierBasisSThm}
Let $X$ be a $d_\rho \times d_\rho$ matrix with a $1$ in the $i,j$ position and $0$ elsewhere. For idempotents $a,b\in D_k$, let $E_{b,a}$ be an $r_k \times r_k$ matrix with a $1$ in the $b,a$ position and $0$ elsewhere. The inverse image in the isomorphism \textup{(\ref{FourierIsomS})} of the element of $\bigoplus_{\bar \rho \in \Y}M_{d_{\bar \rho}}(\C)$ that is $E_{b,a}\otimes X$ in the $\bar \rho$ block and $0$ elsewhere is
\[
\ld p_{b} \rd \left(\sum_{x\in G_k}c_{i,j}(x)\ld x \rd \right)
\ld {p_a}^{-1} \rd
.\]

\end{thm}

\begin{proof}
Suppose $\bar \gamma \in \Y$, $\bar \gamma \neq \bar \rho$, and $\gamma \in \IRR(G_k)$. Then
\[
\bar \gamma \left( \ld p_{b} \rd \left(\sum_{x\in G_k}c_{i,j}(x)\ld x \rd \right)
\ld {p_a}^{-1} \rd \right) = 0
\]
(where $0$ indicates the zero matrix), because
\begin{align*}
\bar \gamma \left(\sum_{x\in G_k}c_{i,j}(x)\ld x\rd \right)
&= E_{e_k,e_k}\otimes\left(\sum_{x\in G_k}c_{i,j}(x)\gamma(x)\right) \\&= E_{e_k,e_k}\otimes 0 \\&= 0. 
\end{align*}
Suppose now that $\bar\gamma\in\Y$ 
 and $\gamma \in \IRR(G_j)$ with $j\neq k$. Then
\[
\bar \gamma \left( \ld p_{b} \rd \left(\sum_{x\in G_k}c_{i,j}(x)\ld x \rd \right)
\ld {p_a}^{-1} \rd \right) = 0
\]
because
\begin{align*}
\bar \gamma \left(\sum_{x\in G_k}c_{i,j}(x)\ld x\rd \right)
&=
\sum_{x\in G_k}c_{i,j}(x)\bar\gamma(\ld x \rd)\\ &= \sum_{x\in G_k}c_{i,j}(x) [0] \\&= 0.
\end{align*}
Finally,
\begin{align*}
\bar \rho & \left(\ld p_{b} \rd \left(\sum_{x\in S_k}c_{i,j}(x)\ld x \rd \right)
\ld {p_a}^{-1} \rd \right) \\ 
=&\bar\rho\left((p_{p_b {p_b}^{-1}})^{-1} p_b p_{{p_b}^{-1}p_b} E_{p_b {p_b}^{-1},{p_b}^{-1} p_b}\right) \cdot
\bar\rho\left(\sum_{x\in G_k}c_{i,j}(x)\ld x\rd \right) \cdot
\\
&\bar\rho\left((p_{{p_a}^{-1} p_a})^{-1} {p_a}^{-1} p_{p_a {p_a}^{-1}} E_{{p_a}^{-1} p_a,p_a {p_a}^{-1}}\right) \\ 
=&\bar\rho\left({p_b}^{-1}p_b p_{e_k} E_{b,e_k}\right)\left(E_{e_k,e_k}\otimes\rho\left(\sum_{x\in G_k}c_{i,j}(x)x \right) \right)\bar\rho\left({p_{e_k}}^{-1}{p_a}^{-1} p_a E_{e_k,a}\right)\\
=&\bar\rho\left(e_k p_{e_k} E_{b,e_k}\right) \left(E_{e_k,e_k}\otimes X\right) \bar\rho\left(p_{e_k} e_k E_{e_k,a}\right) \\
=&\bar\rho\left(e_k E_{b,e_k}\right) \left(E_{e_k,e_k}\otimes X\right) \bar\rho\left(e_k E_{e_k,a}\right) \\
=&\left(E_{b,e_k}\otimes I_{d_\rho}\right) \left(E_{e_k,e_k}\otimes X\right) \left(E_{e_k,a} \otimes I_{d_\rho}\right)\\
=&E_{b,a}\otimes X.
\end{align*}
\end{proof}

Note that the Fourier basis element of $\C S$ in Theorem \ref{FourierBasisSThm}, when expressed in terms of the groupoid basis, has nonzero coefficients only for elements $\ld s \rd$ for which $s\in D_k$, as $p_b\in D_k$, $p_a^{-1}\in D_k$, and $x\in D_k$ for all $x\in G_k$. 

For $\C R_n$, given a Fourier basis $B=\{\sum_{\sigma\in S_k}b_i(\sigma)\sigma\}_{i=1}^{|B|}$ for an isotypic subspace $V^\lambda$ for $\C S_k$, where $\lambda$ is a partition of $k$, we obtain a basis for the corresponding isotypic subspace $V^\lambda$ of $\C R_n$ by forming the products
\[
\ld p_{b}\rd 
\left( \sum_{\sigma\in S_k}b_i(\sigma)\ld \sigma \rd \right)
\ld p_{a}^{-1} \rd,
\]
as $a$ and $b$ range over the size-$k$ subsets of $\{1,2,\ldots,n\}$ (where $p_b$ is the unique order-preserving bijection from $\{1,2,\ldots,k\}$ to $b$ and $p_a^{-1}$ is the unique order-preserving bijection from $a$ to $\{1,2,\ldots, k\}$) and the term in the middle ranges over the elements of $B$.

We can now describe the isotypic projections of an element $f\in \C S$ in terms of isotypic projections in the $\C G_k$.

\begin{thm}
\label{ThmUnderstandingProjections}
Let $a,b\in D_k$ be idempotent. Let $f\in \C S$ have the form
\[
f=\sum_{s\in S}f(s)\ld s \rd,
\]
where $f(s)=0$ unless $s\in D_k$, $s^{-1}s=a$, and $ss^{-1}=b$. Let $f_G$ be $f$ viewed as an element of $\C G_k$, i.e.,
\[
f_G=\sum_{s\in D_k:ss^{-1}=b,s^{-1}s=a}f(s) p_b^{-1}sp_a.
\]
Let $\IRR(G_k)=\{\rho_1,\ldots,\rho_q\}$. For $i\in \{1,\ldots,q\}$, denote the isotypic subspace of $\C G_k$ corresponding to $\rho_i$ by $W_i$ and the isotypic subspace of $\C S$ corresponding to $\bar{\rho_i}$ by $V_i$.
Let $f_G^1,\ldots,f_G^q$ denote the projections of $f_G$ onto $W_1,\ldots,W_q$, and suppose
\[
f_G^i=\sum_{g\in G_k}c_i(g)g.
\]
Then
\[
f=\sum_{i=1}^q \ld p_b \rd f_G^i \ld p_a^{-1} \rd,
\]
$\ld p_b \rd f_G^i \ld p_a^{-1} \rd$ is the projection of $f$ onto $V_i$, and
\[
\ld p_b \rd f_G^i \ld p_a^{-1} \rd = \ld p_b \rd 
\left( \sum_{g\in G_k}c_i(g)\ld g \rd\right)
\ld p_a^{-1} \rd.
\]
\end{thm}

In general, if $f\in \C S$, then $f$ may be written as a sum of elements of $\C S$ of the form in the hypothesis of this theorem. We can then understand the projections of $f$ in terms of the projections of these elements.

\begin{proof}[Proof of Theorem \ref{ThmUnderstandingProjections}]
By the isomorphism in Theorem \ref{BigThm}, we can write
\[
f=\sum_{g\in G_k}f(p_bgp_a^{-1})\ld p_b g p_a^{-1} \rd
=\ld p_b \rd
\left(\sum_{g\in G_k}f(p_bgp_a^{-1}) \ld g\rd \right)
\ld p_a^{-1} \rd
\]
and
\[
f_G=\sum_{g\in G_k}f(p_bgp_a^{-1})g.
\]
Proposition 4.3 of \cite{Steinberg2} states that, for $s,t\in S$,
\begin{equation}
\ld s \rd t = 
\begin{cases}
\ld s t \rd & \textup{if } s^{-1}s \leq tt^{-1},\\
0 & \textup{otherwise}.
\end{cases}
\label{SteinbergInteractionMult}
\end{equation}
From this we see that
\[
f= \ld p_b \rd f_G \ld p_a^{-1}\rd,
\]
and since
\[
f_G=f_G^1+\cdots +f_G^q,
\]
we have
\[
f=\ld p_b \rd f_G^1 \ld p_a^{-1} \rd + \ld p_b \rd f_G^2 \ld p_a^{-1} \rd+ \cdots + \ld p_b \rd f_G^q \ld p_a^{-1} \rd.
\]
Next, since
\[
f^i_G = \sum_{g\in G_k} c_i(g)g,
\]
applying (\ref{SteinbergInteractionMult}) to $\ld p_b \rd f_G^i \ld p_a^{-1} \rd$ yields that
\[
\ld p_b \rd f_G^i \ld p_a^{-1} \rd=
\ld p_b \rd 
\left(\sum_{g\in G_k} c_i(g)\ld g\rd\right)
\ld p_a^{-1} \rd.
\]
Finally, if we choose a Fourier basis of $\C G_k$ and write the $f_G^i$ in terms of this Fourier basis, equation 
(\ref{SteinbergInteractionMult}) and Theorem \ref{FourierBasisSThm} say that for all $i$, $\ld p_b \rd f_G^i \ld p_a^{-1} \rd \in V_i$. That is, $\ld p_b \rd f_G^i \ld p_a^{-1} \rd$ is the projection of $f$ onto $V_i$, as claimed.
\end{proof}

Finally, we give an inner product on $\C S$ under which its isotypic subspaces are mutually orthogonal.

\begin{thm}
\label{OrthogInnerProductS}
Let $\ip{\cdot,\cdot}$ be the sesquilinear form on $\C S$ induced by, for $s,t\in S$,
\[
\ip{\ld s \rd,\ld t \rd} = \begin{cases}
1 & \textup{if } s=t, \\
0 & \textup{otherwise.}\end{cases}
\]
Then, with respect to this inner product, the isotypic subspaces of $\C S$ are mutually orthogonal.\end{thm}

\begin{proof}
By linearity, it suffices to show that $\ip{v,v'}=0$ in the case that $v$ and $v'$ are Fourier basis elements of $\C S$ in distinct isotypic subspaces. 
We may assume that $v$ and $v'$ are part of a Fourier basis for $\C S$ according to $\Y$. Let $\bar\rho,\bar\rho'\in \Y$. Let $v\in V_{\bar\rho}$ and $v'\in V_{\bar\rho'}$, with $V_{\bar\rho} \neq V_{\bar\rho'}$ (and hence $\bar\rho\neq\bar\rho'$).

We know that
\[
\C G_k = \bigoplus_{\rho \in \IRR(G_k)}W_\rho
\]
where $W_\rho$ is the sum of all the irreducible submodules of $\C G_k$ isomorphic to the representation $\rho$. If $w\in W_\rho, w' \in W_{\rho '},$ and $W_\rho \neq W_{\rho '},$ then under the inner product $[\cdot,\cdot]$ on $\C G_k$ defined by
\[
[w,w']=[\sum_{s\in G_k}w(s)s,\sum_{s\in G_k}w'(s)s] = \sum_{s\in G_k} w(s)\overline{w'(s)},
\]
it follows from the discussion in Chapter 2 of \cite{Serre} that we have $[w,w'] = 0$.

Now, suppose that ${\bar\rho}$ is the extension of $\rho \in \IRR(G_k)$ and that ${\bar\rho'}$ is the extension of $\rho' \in \IRR(G_j)$. By Theorem \ref{FourierBasisSThm}, when written in terms of the groupoid basis, $v$ contains nonzero coefficients only for the elements $\ld s\rd$ such that $s \in D_k$, and $v'$ contains nonzero coefficients only for the elements $\ld s\rd$ such that $s \in D_j$. Thus, if $k \neq j$, we have $\ip{v,v'} = 0$. Suppose then that $k=j$. By Theorem \ref{FourierBasisSThm}, we have
\begin{align*}
v =& \ld p_b \rd   \left(\sum_{s \in G_k} v(s)\ld s\rd \right)   \ld {p_a}^{-1} \rd, \\
v'=& \ld p_{b'} \rd \left( \sum_{s \in G_k} v'(s)\ld s\rd  \right) \ld {p_{a'}}^{-1} \rd,
\end{align*}
for $b,a,b',a'$ some idempotents in in $D_k$, and
\[
\sum_{s \in G_k} v(s)s \in W_\rho, \, 
\sum_{s \in G_k} v'(s)s \in W_{\rho '}
\] 
some Fourier basis elements for $\C G_k$. 

If $a \neq a'$ or $b \neq b'$, it is apparent that $\ip{v,v'} = 0$, so suppose further that $a=a'$ and $b=b'$.

Now, since ${\bar \rho} \neq {\bar \rho '}$ and $k=j$, we have $\rho \neq \rho'$, and we therefore note that
\[
[\sum_{s \in G_k} v(s)s, \sum_{s \in G_k} v'(s)s] = 0.
\]
Now, we have
\[
\ip{v,v'} = \sum_{s\in G_k} \sum_{t \in G_k} v(s)\overline{v'(t)} \ip{\ld p_b s {p_a}^{-1} \rd , \ld p_b t {p_a}^{-1} \rd},
\]
and, since $s,t\in G_k$, $\ld p_b s {p_a}^{-1} \rd = \ld p_b t {p_a}^{-1} \rd$ if and only if $s = t$, so
\[
\ip{v,v'} = \sum_{s \in G_k} v(s) \overline{v'(s)}
=[\sum_{s \in G_k} v(s),\sum_{s \in G_k} v'(s)]
=0.
\]
\end{proof}

%
%
%
%


\bibliographystyle{plain}
\bibliography{RookPRDbib}

\end{document}